\documentclass[a4paper,11pt]{article}
\usepackage{amssymb,latexsym,amsmath,amsfonts,amsthm}
\usepackage{enumerate}
\usepackage{graphicx}
\usepackage{hyperref}
\usepackage{natbib}

\usepackage{setspace}
\newtheorem{thm}{Theorem}[section]
\newtheorem{cor}[thm]{Corollary}
\newtheorem{lem}[thm]{Lemma}
\newtheorem{prop}[thm]{Proposition}

\theoremstyle{definition}
\newtheorem{defn}[thm]{Definition}
\newtheorem{rem}[thm]{Remark}

\numberwithin{equation}{section}

\usepackage{color}

\newcommand{\R}{\mathbb{R}}
\renewcommand{\P}{\mathbb{P}}
\newcommand{\E}{\mathbb{E}}
\newcommand{\N}{\mathbb{N}}
\newcommand{\Z}{\mathbb{Z}}
\newcommand{\Tu}{T_u}

\newcommand{\diag}{\mathrm{Diag}}

\newcommand{\vect}[2]{\left( \begin{array}{c} #1 \\ #2 \end{array} \right)}

\newcommand{\vectthree}[2]{\left( \begin{array}{c} #1 \\ \vdots \\ #2 \end{array} \right)}

\newcommand{\ob}[1]{{#1}^\ast}
\newcommand{\bb}{\boldsymbol}
\renewcommand{\mid}{\, \Big| \,}

\newcommand{\anorm}[1]{\left\| #1 \right\|_{\bb \alpha}}

\newcommand{\supp}{\mathrm{supp}}

\newcommand{\ie}{{\em i.e.}}
\newcommand{\eg}{{\em e.g.}}

\newcounter{stepnumber}[thm]
\newcommand{\step}{\stepcounter{stepnumber} \textbf{Step \arabic{stepnumber}.~}}

\newcommand{\eqdist}{\stackrel{\mathrm{law}}{=}}

\begin{document}

\title{Asymptotic Independence  {\em ex machina} - Extreme Value Theory for the Diagonal SRE Model}

\author{Sebastian Mentemeier\thanks{Universität Hildesheim, Germany. mentemeier@uni-hildesheim.de}, Olivier Wintenberger\thanks{Sorbonne Universit\'e Paris, France. olivier.wintenberger@upmc.fr}}

\maketitle

\begin{abstract}
We consider multivariate stationary processes $(\boldsymbol{X}_t)$ satisfying a stochastic recurrence equation of the form
$$ \boldsymbol{X}_t= \mathbb{ M}_t \boldsymbol{X}_{t-1} + \boldsymbol{Q}_t,$$
where  $(\boldsymbol{Q}_t)$ are iid random vectors and 
$$
\mathbb{M}_t=\mathrm{Diag}(b_1+c_1 M_t, \dots, b_d+c_d M_t)
$$ 
are iid diagonal matrices and $(M_t)$ are iid random variables. 
We obtain a full characterization of the multivariate regular variation properties of $(\boldsymbol{X}_t)$, proving that coordinates $X_{t,i}$ and $X_{t,j}$ are asymptotically independent even though all coordinates rely on the same random input $(M_t)$. We describe extremal properties of $(\boldsymbol{X}_t)$ in the framework of vector scaling regular variation.
Our results are applied to some multivariate autoregressive conditional heteroskedasticity (BEKK-ARCH and CCC-GARCH) processes.
\end{abstract}


\medskip

multivariate regular variation, non-standard regular variation

\section{Introduction}

We consider multivariate stationary processes $(\boldsymbol{X}_t)$, satisfying a diagonal Stochastic Recurrence Equation (SRE) of the form
\begin{equation}\label{eq:SRE}
\bb X_t=\mathbb{M}_t \bb X_{t-1}+\bb Q_t,\qquad t\in\Z,
\end{equation}
where $(\mathbb{M}_t)$ is an iid sequence of matrices such that for non-negative coefficients $b_i,$ $c_i$, $1\le i\le d$, 
\begin{equation}\label{eq:diagmt}
\mathbb{M}_t=\mathrm{Diag}(b_1+c_1 M_t, \dots, b_d+c_d M_t)\,,\qquad t\in \Z\,,
\end{equation}
and $(\bb Q_t)$ is an iid sequence of $\R^d$ random vectors with marginals $Q_{t,i}$, $1\le i\le d$, independent of iid random variables $(M_t)$.  Stationary solutions of SRE have attracted a lot of research in the past few years, see  \cite{Buraczewski.etal:2016} and references therein. However, in the present setting of diagonal matrices, only marginal tail behavior has been investigated so far using the result of the seminal paper of \cite{Goldie1991}.  Moreover,  serial extremal dependence of the marginal sequences $(X_{t,i})_{t\in\Z}$ for any $1\le i\le d$ is well known since the pioneer work of \cite{deHaan1989extremal}. Under Assumptions \eqref{eq:A1} -- \eqref{eq:A6} that are introduced in Section \ref{sec:univariate}, applying the he Kesten-Goldie-Theorem of \cite{Goldie1991,Kesten1973} we know that 
\begin{equation}\label{eq:pareto}
\P(X_{0,i} > x) \sim a_i x^{-\alpha_i},\qquad x\to \infty,
\end{equation}
where $a_i$ is a positive constant and $\alpha_i>0$ is the unique solution of the equation $\E[|b_i+c_iM_0|^{\alpha_i}]=1$. Here and below, $f(x) \sim g(x)$ means that $\lim_{x \to \infty} \frac{f(x)}{g(x)}=1$. We will show that $\alpha_i\neq \alpha_j$ in many situations, for instance when $c_i/c_j> b_i/b_j\ge 1$ and  $c_j>0$  (with the convention $0/0=1$).


\medskip

Diagonal SRE \eqref{eq:SRE} is a very simple  model that may coincide with some classical multivariate GARCH ones. For the specific case $(M_t)$ are iid $\mathcal{N}(0,1)$, $b_i=0$ for all $1\le i\le d$ and $(\bb Q_t)$ are iid  $\mathcal{N}(0,\bb{\Sigma})$ the diagonal SRE coincides with the diagonal BEKK-ARCH(1) model as in \cite{pedersen:wintenberger:2018}.
For $(\bb Q_t)$ degenerated to a constant  the diagonal SRE model coincides with the volatility process of some CCC-GARCH model, see Section \ref{sec:generalmodel} for details. Such diagonal SRE models are very interesting as they generate potentially different  marginal tail indices $\alpha_i>0$. This freedom is not offered by 
the general CCC-GARCH model the marginals of which have the same tail index, as discussed in \cite{Buraczewski.etal:2016} and \cite{starica:1999}. This feature is important for modeling: Heavy tailed data, such as in finance, may exhibit different tail indices indicating different magnitude in the responses during financial crisis.   
 
 \medskip
We provide in this paper the \textbf{joint extremal behaviour}, {\em i.e.}, multivariate regular variation of $(\bb X_t)$ and the interplay between marginals that have distinct tail indices. As an example, consider the case of the bivariate case  $(X_{0,1}, X_{0,2})$ such that $c_2/c_1> b_2/b_1\ge 1$ and $c_1>0$ and positive $M$ and $Q$. Then $\alpha_1>\alpha_2$ and our first main result in \textbf{Section \ref{sec:asind}} states
%
that $X_{0,1}$ and $X_{0,2}$ are asymptotically {\em independent} in the sense that
\begin{equation*}
\lim_{x\to \infty}\P(X_{0,1}>x^{\frac{1}{\alpha_1}} \, | \, X_{0,2}>x^{\frac{1}{\alpha_2}})=\lim_{x\to \infty}\P( X_{0,2}>x^{\frac{1}{\alpha_2}}\, | \, X_{0,1}>x^{\frac{1}{\alpha_1}})=0.
\end{equation*}
This result remains true also when $Q_1=Q_2$. Thus, even though $X_{0,1}$ and $X_{0,2}$ are perfectly dependent in the sense that all their randomness comes from the same random variables, extremes never occur simultaneously in these marginals. 
This result together with the marginal regular variation property \eqref{eq:pareto} allow us to derive that the random vector $(X_{0,1}, X_{0,2})$ is  {\em Vector Scaling Regular Variation (VSRV)}  in the sense of \cite{pedersen:wintenberger:2018}, see \textbf{Section \ref{sec:VSRV}}. This notion is a simple alternative of the non-standard regular variation of \cite{resnick:2007} suited for Pareto equivalent marginal tails satisfying \eqref{eq:pareto}. The extremal dependence among the marginals is described through the spectral vector $\widetilde{\bb \Theta}_0$ and we will show the following (see Theorem \ref{thm:general case} for the full result):

\begin{thm}\label{th:intro}
Suppose \eqref{eq:A1} -- \eqref{eq:A6} with positive $Q_1, Q_2$ and $M$, and that  $c_2/c_1> b_2/b_1\ge 1$ with $c_1>0$. Then  $\bb X_0$ is  VSRV and its components are asymptotically independent, {\em i.e.} its spectral component $\widetilde{\bb \Theta}_0$ satisfies $$\widetilde{\bb \Theta}_0\in\{(1,0),(0,1)\} \quad \text{a.s.}$$
\end{thm}


\textbf{Section \ref{sec:diageq}} concerns the case where the diagonal terms $m_i$ are identically equal to $m$ and hence the tail indices of  the marginals $X_{0,i}$ are the same. Applying Theorem 1.6 of \cite{buraczewskietal} on the SRE equation \eqref{eq:SRE}  with multiplicative similarity matrix $(b+cM_0) \bb I_d$ \cite{pedersen:wintenberger:2018} derived multivariate regular variation of the process $(X_t)$. We refine this result by characterizing the angular properties of the tail measure.

\medskip

To study the general diagonal SRE where some diagonal elements are identical and others are distinct, we use the notion of VSRV that extend easily from the marginal distribution to stationary Markov chains, see \textbf{Section \ref{sec:VSRV}}.
The joint extremal behaviour is described via a spectral tail process $(\widetilde{\bb \Theta}_t)_{t \ge 0}$, satisfying the recursive equation
\[
\widetilde{\bb \Theta}_t= \mathbb{M}_t\widetilde{\bb \Theta}_{t-1},\qquad t\ge 1
\]
from some initial value $\widetilde{\bb \Theta}_0$.  In \textbf{Section \ref{sec:generalmodel}}, we derive the characterization of $\widetilde{\bb \Theta}_0$, proving asymptotic independence between blocks with different tail indices, and asymptotic dependence within blocks. Since the block structure is stable by multiplication by the diagonal matrix $ \mathbb{M}_t$, we show that the clusters of extremes of different magnitude and tail indices are asynchronous. Thus diagonal multivariate GARCH processes model different magnitude responses  to asynchronous financial crisis.  

We apply this result to characterize the spectral tail measure of the BEKK-ARCH(1) model and the CCC-GARCH model mentioned above. A simulation study illustrates our results.

 
%

\section{Preliminaries}

\subsection{Notation} $\|\cdot\|$ will denote the max-norm  on $\R^d$ and $\|\cdot\|_2$ the euclidean norm.  For vectors, we use bold notation  $\bb x =(x_1, \dots, x_d)$. Operations between vectors or scalar and vector are interpreted coordinate wise, \eg,  $x^{-1/\bb \alpha} = (x^{-1/\alpha_1}, \dots, x^{-1/\alpha_d})$ for positive $x$ and $\bb a \bb b = (a_i b_i)_{1 \le i \le d}$. A notation  that will be used frequently is vector scaling of a sequence of $\R^d$-valued random variables, e.g.
\begin{align*}   x^{-1/\bb \alpha} (\bb X_0, \dots, \bb X_t) ~&=~ \big( x^{-1/\bb \alpha } \bb X_0, \dots,  x^{-1/\bb \alpha} \bb X_t\big) \nonumber \\
 ~&=~ \big( \big( x^{-1/\alpha_i} X_{0,i}\big)_{1 \le i \le d}, \dots, \big( x^{-1/\alpha_i} X_{t,i}\big)_{1 \le i \le d} \big) 
  .\end{align*}
For some potentially distinct $\alpha_1,\ldots,\alpha_d$ we define the following notion of a radial distance: 
\begin{equation*}
\|\bb x\|_{\bb \alpha}= \max_{1\le i\le d} |x_i|^{\alpha_i}  = \| \bb x^{\bb \alpha} \| ,\qquad \bb x= (x_i)_{1\le i\le d}\in \R^d.
\end{equation*}
Here $\bb x^{\bb \alpha}$ denotes the vector $({\rm sign}(x_i)|x_i|^{\alpha_i})_{1\le i\le d}$ in $\R^d$.
We want to stress that $\|\bb x\|_{\bb \alpha}$ is neither homogeneous nor does it satisfy the triangle inequality  for general values of $\alpha_1, \dots, \alpha_d$. Thus, it is not a (pseudo-)norm but it will provide a meaningful scaling function. Note that $\bb x \mapsto \| \bb x\|_{\bb \alpha}$ is a continuous function and is $ 1/\bb\alpha$-homogeneous in the following sense:
\begin{equation*}
\| \lambda^{1/\bb \alpha} \bb X_0\|_{\bb \alpha} ~=~ \max_{1 \le i \le d} \left|{\lambda^{1/\alpha_i}} {X_{0,i}} \right|^{\alpha_i} ~=~ \lambda \| \bb X_0\|_{\bb \alpha}
\end{equation*} The components of the vector 
\begin{align*}
\|\bb X_0\|_{\bb \alpha}^{-1/\bb \alpha} \bb X_t ~=~ \big( \|\bb X_0\|_{\bb \alpha}^{-1/\alpha_i} X_{t,i} \big)_{1 \le i \le d} 
\end{align*}
have $ \| \cdot\|_{\bb \alpha}$ and max-norm equal to one when $t=0$ thus belongs to ${\cal S}_\infty^{d-1}=\{\bb x\in \R^d;\,\|\bb x\|_{\bb \alpha}=1\}$ the max-norm-unit sphere.


\subsection{The univariate marginal SRE and the assumptions}\label{sec:univariate}
Due to the diagonal multiplicative term in \eqref{eq:SRE}, the marginals of $\bb X_t=(X_{t,1},\ldots,X_{t,d})^\top$ are satisfying the univariate marginal SREs $X_{t,i}=(b_i+c_i M_t) X_{t-1,i}+Q_{t,i}$, $t\in \Z$ for $1\le i\le d$.
We work under the following set of assumptions that implies the ones of \cite{Goldie1991} on  the marginal SREs. 
Denoting by $(M, \bb Q)$ a generic copy of $(M_t, \bb Q_t)$, we assume that for all $1 \le i \le d$,
\begin{equation}\label{eq:A1} \tag{A1}
\E\big[ \log \big|b_i+c_i M|\big] <0 .   \end{equation} 
This guarantees that the Markov chain $(\bb X_t)_{t \in \N}$ has a unique stationary distribution. It is given by the law of the random variable
\begin{equation} \label{eq:defV} \bb X ~=~ \vectthree{X_1}{X_d} ~:=~ \sum_{k=1}^\infty \mathbb{M}_1 \cdots \mathbb{M}_{k-1} \bb Q_k. \end{equation}

%

We further assume that there exist positive constants $\alpha_1, \dots, \alpha_d$ such that for $1 \le i \le d$
\begin{equation} 
\label{eq:A2} \tag{A2}
\E \big[|b_i+c_i M|^{\alpha_i}\big] ~=~1.
\end{equation}
Given these $\alpha_1, \dots, \alpha_d$, we assume for $1 \le i \le d$
\begin{equation} \label{eq:A3}\tag{A3}
 \E \big[|M|^{\alpha_i+\epsilon} \big] < \infty, \quad \E \big[ \|\bb Q\|^{\alpha_i+\epsilon} \big] < \infty 
\ \text{ for some $\epsilon >0$} .
\end{equation}
Of course, it suffices to check this condition for the maximal $\alpha_i$. We also need the technical assumption that
\begin{equation} \label{eq:A4} \tag{A4}
\text{the distributions of  $\log |b_i+c_i M|$ are non-arithmetic for all $1\le i\le d$.}
\end{equation}
Finally, to avoid degeneracy, we require for $1 \le i \le d$ that
\begin{equation} \label{eq:A5} \tag{A5} \P((b_i+c_iM)x+ Q_i= x)<1 \quad \text{ for all } x \in \R,  \text{ and } \P(Q_i>0)>0\end{equation}
For all pairs $1 \le i,j \le d$ such that $\alpha_i > \alpha_j$, we will require that
\begin{equation} \label{eq:A6} \tag{A6}\lim_{u\to \infty} \log(u) \, \P\Big( \dfrac{|Q_j|}{|Q_i|}>u^\varepsilon \Big)=0\, \quad \text{ for all $\epsilon >0$}. \end{equation}
This last assumption is specific for our results. We note in addition that the stationarity condition  $\eqref{eq:A1}$ can be deduced from \eqref{eq:A2} as soon as $M$ is not constant a.s. (which is implied by \eqref{eq:A4}), see the comments after Theorem 2.4.4 in \cite{Buraczewski.etal:2016}.

Given \eqref{eq:A1}--\eqref{eq:A5},  an application of the Kesten-Goldie-Theorem of \cite{Goldie1991,Kesten1973} yields the existence of a Pareto tail equivalent stationary distribution, i.e. the equivalence in \eqref{eq:pareto} is met. The positivity of $a_i$ follows by non trivial classical arguments. If $P(b_i+c_iM<0)>0$, then positivity of $a_i$ is proved in \cite[Theorem 4.1]{Goldie1991}. If $b_i+c_iM >0$ a.s., then additional arguments are needed: Assumptions \eqref{eq:A1} and \eqref{eq:A2} together imply that there are $m,m'$ in the support of $M$, such that $b_i+c_im<1$, $b_i+c_im'>1$. By \eqref{eq:A5}, there is $q_i>0$ in the support of $Q_i$. Since $Q_i$ and $M$ are assumed to be independent, we have that $(b_i+c_im,q_i)$ and $(b_i+c_im',q_i)$ are in the support of $(b_i+c_iM,Q_i)$. Then \cite [Proposition 2.5.4]{Buraczewski.etal:2016} yields that the support of $X_i$ is unbounded at $+\infty$, which together with \cite[Theorem 2.4.6]{Buraczewski.etal:2016} implies that $a_i$ is positive.

\section{Vector Scaling Regular Variation Markov chains}\label{sec:VSRV}

\subsection{Regular variation and the tail process}

Let $(\bb X_t)\in \R^d$ be a stationary time series. Its regular variation properties are defined in different ways. The most usual way is to define the tail process as in \cite{basrak:segers:2009}.  
\begin{defn}
The stationary time series $(\bb X_t)$ is regularly varying if and only if $\|\bb X_0\|$ is regularly varying and for all $t \ge 0$ there exist  weak limits
\[
\lim_{x\to \infty }\P\Big(\|\bb X_0\|^{-1}(\bb X_0,\ldots,\bb X_t)\in \cdot  \mid \|\bb X_0\|>x\Big) = \P\big((\bb \Theta_0,\ldots,\bb \Theta_t)\in \cdot\big)\,.
\]
\end{defn}
By stationarity and using Kolmogorov consistency theorem one can extend the trajectories $(\bb \Theta_0,\ldots,\bb \Theta_t)$ into  a process $(\bb \Theta_t)$ called the {\em spectral tail process}. To be  stationary regularly varying time series does not depend on the choice of the norm. We work with the max-norm in the following for convenience.

\subsection{Non-standard Regular Variation}\label{sec:nsrv}

If there exists $1 \le i \le d$ such that 
\[
\P(|X_{0,i}|> x) = o\big(\P(  \|\bb X_0\|>x)\big)\, ,\qquad x\to \infty\,,
\]
then the marginals of $\bb X_0$ are {\em not tail equivalent}. In this case, the notion introduced above is not suitable, since then the corresponding coordinate of the spectral tail process is degenerated, \ie, $\Theta_{0,i}=0$ a.s. Hence, information about extremes in this coordinate is lost. 

To circumvent this issue, the notion of {\em non-standard regular variation} was introduced (see \cite{resnick:2007} and reference therein). It is based on a standardization of the coordinates which holds as follows.   Assume that marginals are positive and (one-dimensional) regularly varying with possibly different tail indices $\alpha_i$ and cdf $F_i$, $ 1 \le i \le d$. Then non-standard regular variation holds if and only if 
$$ \lim_{x \to \infty} x \cdot \P \big( x^{-1 }\widetilde{ \bb X_0} \in \cdot \big)$$
exists in the vague sense, where the standardized vector $\widetilde{ \bb X_0}$ is defined as
\begin{equation*}
\widetilde{ \bb X_0} ~=~ (1/(1-F_i(X_{0,i})))_{1\le i\le d}\,.
\end{equation*}
Following \cite[Theorem 4]{dehaan:resnick:1977}, we note that $\widetilde{ \bb X_0}$ is regularly varying in the classical sense, i.e. $\|\widetilde{ \bb X_0}\|$ is regularly varying with tail index $1$ and there exists an angular measure which is the weak limit of
\[
\lim_{x\to \infty}\P\Big(\|\widetilde{ \bb X_0}\|^{-1}\widetilde{ \bb X_0}\in\cdot\mid \|\widetilde{ \bb X_0}\|>x\Big)\,.
\]
Note that the standardization is made so that all  coordinates $\widetilde{ \bb X_0}$ of are tail equivalent
\[
\P\big( \widetilde{ X_{0,i}}>x\big)\sim x^{-1}\,, \qquad x\to \infty\,,\qquad 1\le i\le d\,. 
\]

\subsection{Vector Scaling Regular Variation}

When dealing with time series such as diagonal SRE, temporal dependencies between extremes are of particular interest. As it turns out, neither of the notions discussed above is fully adequate for the investigation of these. Indeed, the SRE representation \eqref{eq:SRE} of the diagonal BEKK-ARCH(1) model appeals for an analysis of the serial extremal dependence directly on  $(\bb X_t)$  rather than on a standardized version. For SRE Markov chains such as \eqref{eq:SRE}, it has been shown by \cite{janssen:segers:2014}  that the spectral tail process satisfies the simple recursion
\begin{equation}\label{eq:simplerecursion}
\bb \Theta_t=\mathbb{M}_t\bb \Theta_{t-1},\qquad t\ge 1. 
\end{equation}
This multiplicative property has nice consequences and allows to translate the properties of multiplicative random walks to the extremes of multivariate time series. However, the degeneracy of the coordinates with lower tails discussed in Section \ref{sec:nsrv} propagates through time; If $\Theta_{0,i}=0$ a.s. then $\Theta_{t,i}=0$ a.s. as well for any $t\ge 1$. On the other hand, the standardized version has a more complicated angular measure (See Proposition \ref{prop:nonst}) and does not satisfy an SRE. Thus its serial extremal dependence is less explicit than the simple recursion \eqref{eq:simplerecursion}; see \cite{perfekt1997extreme} for details.

In order to treat the temporal dependence of the stationary solution $(\bb X_t)$, we will use the notion of Vector Scaling Regular Variation (VSRV) introduced in \cite{pedersen:wintenberger:2018} as follows:
\begin{defn}[VSRV]
A stationary time series $(\bb X_t)$ is VSRV of order $\bb \alpha=(\alpha_1,\ldots,\alpha_d)$ if $\P(|\bb X_{0,i}|>x)\sim a_ix^{\alpha_i} $ with $a_i>0$, $1\le i\le d$, $\|\bb X_0\|_{\bb\alpha}$ is regularly varying  and there exists  weak limits
\begin{equation}\label{eq:yt}
\lim_{x\to \infty}\P\Big(\|\bb X_0\|_{\bb \alpha}^{-1/{\bb \alpha}}  (\bb X_{0},\ldots,\bb X_{t}) \in \cdot \mid \|\bb X_0\|_{\bb \alpha}>x\Big)=\P\big((\widetilde{\bb \Theta}_0,\ldots,\widetilde{\bb \Theta}_t)\in\cdot\big)\,,
\end{equation}
for any $t\ge 0$.
\end{defn}
Note that a VSRV time series $(\bb X_t)$ with indices $\alpha_1,\ldots, \alpha_d$ is such that  $(\bb X_t^{\bb \alpha})$ is regularly varying with tail index $1$.  When $\alpha_1=\cdots=\alpha_d$ then  $(\bb X_t)$ is regularly varying and $(\widetilde{\bb \Theta}_t)$ coincides with the spectral tail process $({\bb \Theta}_t)$. It is one advantage of considering VSRV as it extends the useful approach of \cite{basrak:segers:2009} to GARCH models where the marginals $X_{0,i}$ have distributions $F_i$ with different tail indices. 
%

The next proposition shows that indeed any positive VSRV random vector $\bb X_0\in \R^d$ is also non-standard regularly varying.  
\begin{prop}\label{prop:nonst}
Let $\bb X_0$ be a positive VSRV random vector of order $\alpha=(\alpha_1, \dots, \alpha_d)$. Then $\bb X_0$ is non-standard regularly varying and the angular measure is given by
$$
\dfrac{\E\big[  \big\|\bb a^{-1} \widetilde{\bb \Theta}_{0}^{\bb \alpha}\big\| \; {\bb 1}\big(\big\|\bb a^{-1} \widetilde{\bb \Theta}^{ \bb \alpha}_{0}\big\|^{-1}\bb  a^{-1} \widetilde{\bb \Theta}_{0}^{ \bb \alpha}  \in \cdot \big)\big]}{\E\big[  \big\|\bb a^{-1}  \widetilde{\bb \Theta}_{0}^{\bb \alpha}\big\|\big]}\,,
$$
where $\bb a=(a_1, \dots, a_d)$.
\end{prop}
 We remark that the angular measure of  $\bb X_0$ is completely determined by the spectral tail process $(\widetilde{\bb \Theta}_{t})$.  However its expression   is intricate because of the different marginal standardizations $\bb c$ whereas we will derive explicit expressions of $(\widetilde{\bb \Theta}_t)$ for many Markov chains in Section \ref{sec:vsrvmc}. We emphasize that this simplicity is the main motivation for introducing the notion of VSRV rather than using the more general notion of non-standard regular variation.
\begin{proof}
The standardized vector 
$$\widetilde{\bb X}_0= {\bb a}^{-1}{\bb X}_{0}^{\bb \alpha}
$$ has marginal tails equivalent to the standard Pareto marginally distributed vector $\big(1/(1- F_i(X_{0,i}))\big)_{1\le i\le d}$. Moreover $\|\bb X_0\|_{\bb \alpha}$ tail is  Pareto equivalent with tail index $1$ since an union bound yields
$$
\P(|X_{0,1}|^{\alpha_1}>x)\le\P(\|\bb X_0\|_{\bb \alpha}>x)\le \sum_{i=1}^d  \P(|X_{0,i}|^{\alpha_i}>x)
$$
and a sandwich argument concludes. Thus
$\|\widetilde{\bb X}_0\|$ is also regularly varying because, denoting $a_*=\min_{1\le i\le d}a_i$ and $a_0=\lim x^{-1}\P(\|\bb X_0\|_{\bb \alpha}>x)$, we have
\begin{align*}
\P( \|\widetilde{\bb X}_0\|>x \big)&=\P\big( \|{\bb a}^{-1}{\bb X}_{0}^{\bb \alpha}\|>x , a_{*}^{-1}\|{\bb X}_{0}^{\bb \alpha}\|>x\big)\\
&= \P\big( \|{\bb a}^{-1}{\bb X}_{0}^{\bb \alpha}\|>x |\, \|{\bb X}_{0}^{\bb \alpha}\|>xa_{*}\big)\P(\|{\bb X}_{0}^{\bb \alpha}\|>xa_{*})\\
&\sim\P\Big(\dfrac{ \|{\bb a}^{-1}{\bb X}_{0}^{\bb \alpha}\|}{\|{\bb X}_{0}^{\bb \alpha}\|}>\dfrac x{\|{\bb X}_{0}^{\bb \alpha}\|} |\, \|{\bb X}_{0}^{\bb \alpha}\|>xa_{*}\Big)a_0a_{*}^{-1}x^{-1}\\
&\sim\P\Big( \Big\|{\bb a}^{-1}\Big(\dfrac{{\bb X}_{0}}{\|{\bb X}_{0}^{\bb \alpha}\|^{1/{\bb \alpha}}}\Big)^{\bb \alpha}\Big\|>\dfrac x{\|{\bb X}_{0}^{\bb \alpha}\|}  |\, \|{\bb X}_{0}^{\bb \alpha}\|>xa_{*}\big)a_0a_{*}^{-1}x^{-1}\\
&\sim\P\big( \|{\bb a}^{-1}\widetilde{\bb \Theta}_{0}^{\bb \alpha}\|>a_{*} Y^{-1} \big)a_0a_{*}x^{-1}\\
&\sim \E[\|{\bb a}^{-1}\widetilde{\bb \Theta}_{0}^{\bb \alpha}\|]a_0x^{-1}.
\end{align*} 
We conclude that ${\bb X}_0$ is non-standard regularly varying and the angular measure is  the limit, as $x\to \infty$, of the ratio
\begin{align*}
\P\big(\|\widetilde{\bb X}_0\|^{-1}\widetilde{\bb X}_0 \in\cdot & |\; \|\widetilde{\bb X}_0\|>x \big)= \dfrac{\P\big(\|\widetilde{\bb X}_0\|^{-1}\widetilde{\bb X}_0 \in\cdot , \|\widetilde{\bb X}_0\|>x \big)}{\P( \|\widetilde{\bb X}_0\|>x \big)}\\
&= \dfrac{\P\big(\|\widetilde{\bb X}_0\|^{-1}\widetilde{\bb X}_0 \in\cdot , \|\widetilde{\bb X}_0\|>x |\, \|{\bb X}_{0}^{\bb \alpha}\|>xa_{*}\big)}{\P( \|\widetilde{\bb X}_0\|>x |\, \|{\bb X}_{0}^{\bb \alpha}\|>xa_{*}\big)}
\end{align*}
and the desired result follows by definition of $\widetilde{\bb \Theta}_{0}$.
\end{proof}


\subsection{VSRV Markov chains}\label{sec:vsrvmc}

We adapt the work of   \cite{janssen:segers:2014} to our framework. We consider a Markov chain $(\bb X_t)_{t\ge0}$ with values in $\R^d$ satisfying the recursive equation
\begin{equation}\label{eq:mc}
\bb X_t=\bb \Phi(\bb X_{t-1},Z_t),\qquad t\ge 0\,,
\end{equation}
where $\bb \Phi:\R^d\times {\cal E}\mapsto \R^d$ is measurable and $(Z_t)$ is an iid sequence taking values in a Polish space $\cal E$.
 We work under the following assumption, which is the vector scaling adaptation of  \cite[Condition 2.2]{janssen:segers:2014}. As above, we fix in advance the positive indices $\alpha_1,\ldots,\alpha_d$. 
\medskip

{\bf VS Condition for Markov chains:} {\em There exists a measurable function $\bb \phi:{\cal S}_\infty^{d-1}\times {\cal E}\mapsto \R^d$ such that, for all $e\in {\cal E}$,
$$
\lim_{x\to \infty} x^{-1/\bb \alpha}\bb \Phi(x^{1/\bb \alpha} \bb s(x),e)\to \bb \phi(\bb s,e)\,,
$$
whenever $\bb s(x)\to \bb s$ in ${\cal S}^{d-1}_\infty$. Moreover, if $\P(\bb \phi(\bb s,Z_0)=0)>0$ for some $\bb s\in {\cal S}^{d-1}_\infty$ then $Z_0\in {\cal W}$ a.s. for a subset ${\cal W}\subset {\cal E}$ such that, for all $e\in {\cal W}$,
$$
\sup_{\|\bb y\|_{\bb \alpha}\le x}\|\bb \Phi(\bb y,e)\|_{\bb \alpha}=O(x)\,\qquad x\to \infty\,. 
$$}

We extend  $\bb \phi$ over $\R^d\times {\cal E}$ thanks to the relation
$$
\bb \phi(\bb v,e)=\begin{cases}\|\bb v\|_{\bb \alpha}^{1/{\bb \alpha}} \bb \phi\Big(\anorm{\bb v}^{-1/\bb \alpha} v,e\Big)\qquad &\text{if}\qquad \bb v\neq 0,\\
0\qquad&\text{if} \qquad \bb v = 0\,.\end{cases}
$$
We have the following result which extends Theorem 2.1 of \cite{janssen:segers:2014}
\begin{thm}
If the Markov chain $(\bb X_t)$ satisfies the recursion \eqref{eq:mc} with $\bb \Phi$ satisfying the VS condition and if the vector $\bb X_0$ is VSRV with  positive indices $\alpha_1,\ldots,\alpha_d$ then 
$(\bb X_t)_{t\ge 0}$ is a VSRV process and its spectral tail process satisfies the relation
\begin{equation*}
\widetilde{\bb \Theta}_t=\bb \phi(\widetilde{\bb \Theta}_{t-1},Z_t)\,,\qquad t\ge 0\,.
\end{equation*}
started from $\widetilde{\bb \Theta}_0$, the spectral component of $\bb X_0$.
\end{thm}
\begin{proof}
The result follows by an application of Theorem 2.1 in \cite{janssen:segers:2014} to the Markov chain $(\bb Y_t)_{t\ge0}=(\bb X_{t}^{\bb \alpha})_{t\ge 0}$. We have ${\bb Y}_0$ regularly varying since $\bb X_0^{\bb\alpha}$ is VSRV. Moreover 
$$
\bb Y_t= \widetilde{\bb \Phi}(\bb Y_{t-1},Z_t)\,,\qquad t\ge 0\,,
$$
with  $\widetilde{\bb \Phi}(x,z)=( {\bb \Phi}(x^{1/\bb \alpha},z))^{\bb \alpha}$. As the VS condition for Markov chain is the  vector scaling version of the condition 2.2. of \cite{janssen:segers:2014} on $\widetilde{\bb \Phi}$ associated to the limit $\widetilde{\bb \phi}((x ,z)) =\bb \phi((x^{1/\bb \alpha},z))^{\bb \alpha}$, i.e. 
$$
\lim_{x\to \infty} x^{-1}\widetilde{\bb \Phi}(x \bb s(x),e)\to \widetilde{\bb  \phi}(\bb s,e)\,
$$
whenever $\bb s(x)\to \bb s$ in ${\cal S}^{d-1}_\infty$. We obtain that the spectral tail process of $(\bb Y_t)_{t\ge0}$ satisfies the recursion 
$$
\bb \Theta^{\bb Y}_t = \widetilde{\bb \phi}(\bb \Theta^{\bb Y}_{t-1},Z_t)\,,\qquad t\ge 1\,.
$$
The desired result follows as $\widetilde{\bb \phi}((x ,z)) =\widetilde{\bb \phi}((x^{1/\bb \alpha},z))^{\bb \alpha}$ and $\widetilde{\bb \Theta}^{\bb \alpha}_t= {\bb \Theta}^{\bb Y}_t$, $t\ge 0$.
\end{proof}
We are specially interested in Stochastic Recurrence Equations (SRE) corresponding to the Markov chains
$$
\bb X_t=\bb \Phi(\bb X_{t-1},(\bb M,\bb Q)_t)=\bb M_t \bb X_{t-1}+\bb Q_t\,, \qquad t\ge 0\,.
$$
In this setting $(\bb M_t)$ are iid random $d\times d$ matrices and $(\bb Q_t)$ iid random vectors in $\R^d$. We have
\begin{prop}\label{prop:SREVSRV}
The SRE Markov chain $(\bb X_t)_{t\ge 0}$ satisfies Condition VS for positive indices $\alpha_1,\ldots,\alpha_d$ if and only if 
$M_{ij}=0$ a.s.~for any $(i,j)$ so that $\alpha_i>\alpha_j$. Then 
$$
\bb \phi\big(\bb s,(\bb M,\bb Q)\big)=\Big(\sum_{j=1}^d M_{ij}\bb 1_{\alpha_i=\alpha_j}s_j\Big)_{1\le i\le d}\,.
$$
\end{prop}
\begin{proof}
As $x\to\infty$ and $\bb s(x)\to \bb s$, we have 
\begin{align*}
&~ \lim_{x \to \infty} x^{-1/\bb \alpha}  \bb \Phi\big ((x^{1/\bb \alpha}) \bb  s(x),(\bb M,\bb Q)\big) ~=~  \lim_{x \to \infty} x^{-1/\bb \alpha}  \Big(\bb M (x^{1/\bb \alpha} \bb s(x)) + \bb Q\Big) \\
&~=~ \lim_{x \to \infty}\Big(\sum_{j=1}^d M_{ij}s(x)_jx^{1/\alpha_j-1/\alpha_i}\Big)_{1\le i\le d}\,.
\end{align*}
Each coordinate converges to $\sum_{j=1}^d M_{ij}1_{\alpha_i=\alpha_j}s_j$ for any $\bb s\in {\cal S}^{d-1}_\infty$ if and only if $M_{ij}=0$ a.s.~for any $(i,j)$ so that $\alpha_i>\alpha_j$.
\end{proof}

\begin{rem}\label{rem:SREVSRV}
In case of distinct $\alpha_i$'s, it means that the dynamic tail process depends only on the diagonal elements of $\bb M$. In general, specifying $\bb M_t$ to be diagonal, we ensure that if $\bb X_0$ is VSRV then the SRE process is VSRV  with
\[
\widetilde{\bb \Theta}_t=\bb M_t \widetilde{\bb \Theta}_{t-1},\qquad t\ge 1\,,
\]
whatever are the positive indices $\alpha_1,\ldots,\alpha_d$.
\end{rem}

\section{The diagonal SRE with distinct coefficients}\label{sec:asind}
In this section we will show that the marginals of the diagonal SRE with distinct coefficients are asymptotically independent. A standard argument reduces the discussion to the bivariate case. We consider the bivariate random recursive process $\bb X_t = \mathbb M_t \bb X_{t-1} +\bb Q_t$, defined by $\bb X_0=0$ and 

%
%
%
%


\begin{equation}
\label{eq:RDE}  \vect{X_{t,1}}{X_{t,2}} ~=~ \left(\begin{matrix}b_1+c_1 M_t & 0\\0& b_2+c_2 M_t\end{matrix}\right) \vect{X_{t-1,1}}{X_{t-1,2}}  + \bb Q_t.
\end{equation}
We assume that $(M_t)_{t \in \N}$ are iid random variables, $(\bb Q_t)_{t \in \N}$ are iid $\R^2$-valued random vectors independent of $(M_t)$ and one of the two following cases:
\begin{align}
\tag{Case I'}\label{case1'}  &b_2 = b_1 = 0, \qquad  c_2 >c_1 >0, \qquad & &\text{$M_t$ is $\R$-valued}  \\
\tag{Case II'}\label{case2'}  &b_2 \ge b_1 >0, \qquad c_2 > c_1 >0, \qquad \frac{c_2}{c_1} \ge \frac{b_2}{b_1}, \qquad & &\text{$M_t >0$ a.s.}
\end{align}
 We assume assumptions \eqref{eq:A1} -- \eqref{eq:A6} to hold for $i=1,2$, which gives in both cases that
$$ \alpha_1 > \alpha_2.$$

Under our assumptions, by the Kesten-Goldie-Theorem of \cite{Goldie1991,Kesten1973} applied to multiplicative factors with $b_i+c_iM$, $i=1,2$, we have for the random variables $X_i$, defined by \eqref{eq:defV}
\begin{equation}
\label{eq:tails} \lim_{u \to \infty} u^{\alpha_1} \P(X_1 >u) = a_1, \qquad \lim_{u \to \infty} u^{\alpha_2} \P( X_2 >u) = a_2
\end{equation}
for constants $a_1, a_2$ which are positive, see Section \ref{sec:univariate} for details.  We are going to prove that
\begin{equation} 
\label{eq:asympIndep} \lim_{u \to \infty} u\, \P\big( X_2 > u^{1/\alpha_2} \, , \, X_1 > u^{1/\alpha_1} \Big) ~=~0. \end{equation}
which by \eqref{eq:tails} is equivalent to the asymptotic independence
$$ \lim_{u \to \infty}  \P\big( X_2 > u^{1/\alpha_2} \, \big| \, X_1 > u^{1/\alpha_1} \Big) ~=~0.$$
of the extremes.

%
%
%
%
%

\subsection{Reduction to the case of nonnegative $M$ and $Q_i$}

From Definition \eqref{eq:defV}, it is obvious that we can bound $X_i$ by the following sums over nonnegative random variables:
$$ X_i ~\le~ \sum_{k=1}^\infty \prod_{\ell=1}^{k-1}|b_i+c_iM_\ell| |Q_{k,i}| ~=:~ \ob{X}_i$$
We notice that $\ob{X}_i$ satisfies the fixed point equation, in distribution, 
$$
\ob{X}_i \eqdist |b_i + c_iM| \ob{X}_i + |Q_i|\,,\qquad i=1,2\, ,
$$
(where $\eqdist$ denotes equality in law between random variables on both sides).
In particular, thanks to \eqref{eq:A1}--\eqref{eq:A4}, the Kesten-Goldie theorem, now used in the case of positive coefficients, applies and yields
\begin{equation}\label{eq:tailsXstar}  \lim_{u \to \infty} u\, \P\Big( \ob{X}_2 > u^{1/\alpha_2} \Big) = a_2^*>0 \qquad \lim_{u \to \infty} u \P \Big( \ob{X}_1 > u^{1/\alpha_1} \Big) = a_1^* >0.
\end{equation} Note that the tail indices $\alpha_1$, $\alpha_2$ remain unchanged thanks to their definition in \eqref{eq:A2}. Since $|X_i| \le X_i^\ast$, $i=1,2$, the result \eqref{eq:asympIndep} will follow from the relation
$$ \lim_{u \to \infty} u\, \P\Big( \ob{X}_2 > u^{1/\alpha_2} \, , \, \ob{X}_1 > u^{1/\alpha_1} \Big) ~=~0.$$

%
%
%

\subsection{Asymptotic independent diagonal SRE}

By the previous discussion, it is enough to consider the following cases
\begin{align}
\tag{Case I} \label{case1} &b_2 = b_1 = 0, \qquad  c_2 >c_1 >0, \qquad & &\text{$M_t$, $Q_i>0$ a.s},  \\
\tag{Case II} \label{case2} &b_2 \ge b_1 >0, \qquad c_2 > c_1 >0, \qquad \frac{c_2}{c_1} > \frac{b_2}{b_1}, \qquad & &\text{$M_t$, $Q_i>0$ a.s}.
\end{align}
We summarize these two cases under the condition $c_2/c_1>b_2/b_1\ge 1$ and $c_1>0$ (with the convention $0/0=1$).
We are going to prove the following result.

\begin{thm}\label{thm:asymptotic independence}
Assume \eqref{eq:A1}--\eqref{eq:A6} for $i=1,2$ with $c_2/c_1>b_2/b_1\ge 1$ and $c_1>0$.
Then we have
$$ \lim_{u \to \infty} u\, \P\big( X_2 > u^{1/\alpha_2} \, , \, X_1 > u^{1/\alpha_1} \Big) ~=~0,$$
i.e., $X_1$ and $X_2$ are asymptotically independent.\end{thm}

The basic tool in the proof is to analyze the behavior of $X_2$ under an exponential change of measure  that favors large values for $X_1$. Namely, we consider the  probability measure $\P^{\alpha_1}$, under which $(M_n)$ is still an iid sequence, but with the new law
$$ 
\P^{\alpha_1}(M \in \cdot): =\E\big[(b_1+c_1 M)^{\alpha_1}1(M\in\cdot)\big].
$$
The law of the sequences $(Q_{n,i})$ remains unchanged and independent of $(M_n)$ under $\P^{\alpha_1}$.

Considering the random variables $$W_{1}:=\log (b_1 +c_1 M) \qquad \text{and} \qquad W_{2}:=\log (b_2+c_2 M),$$
with associated iid sequences $W_{n,i}:=\log(b_i+c_i M_n)$, we denote their respective $\P^{\alpha_1}$-drift by
$$
\mu_{j|1} ~:=~ \E[\log (b_j+c_jM) (b_1+c_1M)^{\alpha_1}] ~=~ \E^{\alpha_1} \big[W_{j}\big],\qquad j=1,2.
$$
We have the following result.

\begin{lem} In both \eqref{case1} and \eqref{case2}, 	it holds that
	\begin{equation}\label{eq:condition} \alpha_2 \mu_{2|1} < \alpha_1\mu_{1|1}. 
	\end{equation} 
\end{lem}

\begin{proof}
	Using Jensen's inequality under the change of measure, we obtain 
	\begin{align*}
	\alpha_2 \mu_{2|1} - \alpha_1 \mu_{1|1} ~=&~ \E \Big[ \log \bigg( \frac{(b_2+c_2M)^{\alpha_2}}{(b_1+c_1M)^{\alpha_1}} \bigg) (b_1+c_1M)^{\alpha_1} \Big] \\
	=&~ \E^{\alpha_1} \Big[ \log \bigg( \frac{(b_2+c_2M)^{\alpha_2}}{(b_1+c_1M)^{\alpha_1}} \bigg)  \Big]\\
	 ~<&~ \log  \E^{\alpha_1} \Big[  \bigg( \frac{(b_2+c_2M)^{\alpha_2}}{(b_1+c_1M)^{\alpha_1}} \bigg) \Big]
	\\ 
	=&~ \log  \E \Big[   \frac{(b_2+c_2M)^{\alpha_2}}{(b_1+c_1M)^{\alpha_1}}  (b_1+c_1M)^{\alpha_1} \Big] ~=~ 0 
	\end{align*}
	The strict inequality holds since $\log$ is strictly convex and the random variable $(b_2+c_2M)^{\alpha_2}/(b_1 +c_1 M)^{\alpha_1}$ is not constant a.s., due to the different exponents and condition \eqref{eq:A4} which implies that $M$ is not degenerate.
\end{proof}

\medskip

\begin{proof}[Proof of Theorem \ref{thm:asymptotic independence}] 
We are going to study partial sums converging to the random variables $X_1$, $X_2$ given by \eqref{eq:defV}, namely
\begin{equation}\label{eq:sumV}
X_{j:m,i} ~:=~ \sum_{k=j+1}^m  \prod_{l=1}^{k-1}(b_i+c_i M_l)Q_{k,i}, \qquad i=1,2.
\end{equation}
We write $X_{n,i}:=X_{0:n,i}$ and observe that $X_i = \lim_{n \to \infty} X_{n,i} = \sup_{n \ge 0} X_{n,i}$ a.s.
Note the distinction between the Markov chain $(X_{t,i})$ (the {\em forward process})  and the almost surely convergent series $(X_{n,i})$ defined above (the {\em backward process}); see  \cite{letac1986contraction}. 

\medskip

\step We gain additional control by introducing the first exit time for $(X_{n,1})$,
$$ {\Tu} := \inf \big\{ n \in \N \, : \, X_{n,1}  > u^{1/\alpha_1} \big\}\,. $$
As $X_i= \sup_{n\ge 0} X_{n,i}$ for $i=1,2$ we have $\{X_1  > u^{1/\alpha_1}\}=\{{\Tu}<\infty\}$. 
By \eqref{eq:tailsXstar} we have 
\begin{equation}\label{eq:kg}
\lim_{u \to \infty} u\cdot \P({\Tu}<\infty) >0.
\end{equation}
Thus, the desired result will follow from the relation 
\begin{equation}\label{eq:condtu}
\lim_{u \to \infty}  \P\big( X_2> u^{1/\alpha_2} \, \big| \, {\Tu}<\infty\big) ~=~0.
\end{equation}

On the set $\{{\Tu} < \infty\}$, it holds
\begin{equation}\label{eq:decomposing V} 
X_2 = X_{\Tu,2} + \prod_{l=1}^{\Tu}(b_2+c_2 M_l) X_{{\Tu}:\infty,2}. 
\end{equation}
The simple inclusion
$$ \{ X_2 > s\} \subset \underbrace{\big\{ X_{{\Tu},2} > u^{1/\alpha_2}/2 \big\}}_ {=:~A _u }\cup \underbrace{\Big\{ \prod_{l=1}^{\Tu}(b_2+c_2 M_l) X_{{\Tu}:\infty,2}> u^{1/\alpha_2}/2 \Big\}}_{=:~B_u}$$
allows us to consider the contributions in \eqref{eq:decomposing V} separately. The following lemma, to be proved subsequently, provides stronger control and is the crucial ingredient for evaluating the contributions of $A_u$ and $B_u$. The proof of this lemma is deferred to the appendix.

\begin{lem}\label{lem:eqcond}
For any $\epsilon >0$, define the set $C_u(\epsilon)$ as the intersection
\begin{multline*}
\Big\{ {\Tu}\le L_u\Big\}\cap \Big\{ {X_{\Tu,1}} \le u^{\tfrac{1 + \epsilon}{\alpha_1}} \Big\}\cap\Big\{\max_{1\le k\le L_u}\dfrac{Q_{k,2}}{Q_{k,1}} \le  u^{\varepsilon/\alpha_1}\Big\}\\
 \cap \Big\{ \sum_{l=1}^{\Tu}(W_{l,2}-W_{l,1})-T_u(\mu_{2|1}-\mu_{1|1})\le \epsilon\Tu\Big\}  \cap \Big\{ \sum_{l=1}^{L_u} W_{l,1} \le \frac{1+\epsilon}{\alpha_1} \log u\Big\}
\end{multline*}
where $L_u:=\log(u)/(\mu_{1|1}\alpha_1)+Cf(u)$, $f(u):= \sqrt{log(u) \cdot \log(\log(u))}$ and $C$ is a (suitably large) constant that can be chosen indepently of $\epsilon$. \\
Then it holds that 
$$\lim_{u\to \infty}\P \big( \big\{X_2>u^{1/\alpha_2} \big\} \cap  C_u(\epsilon)  \, \big| \, {\Tu} < \infty \big)
= \lim_{u\to \infty} \P\big( X_2 > u^{1/\alpha_2} \, \big| \, {\Tu} < \infty\big)
$$
if either of the limits exists.
%
\end{lem}
\setcounter{stepnumber}{1}
\step Considering the event $A_u$, we have, using $b_1\le b_2$ and $c_1 < c_2$ and the controls  provided by $C_u(\varepsilon)$, that
\begin{align}
X_{\Tu,2} ~=&~ \sum_{k=1}^{\Tu}  \prod_{l=1}^{k-1}(b_2+c_2 M_l)Q_{k,2}\nonumber \\
\le&~  \Big(\max_{1\le k\le \Tu}\dfrac{Q_{k,2}}{Q_{k,1}}\Big)\sum_{k=1}^{\Tu}  \prod_{l=1}^{k-1}\dfrac{b_2+c_2 M_l}{b_1+c_1 M_l}(b_1+c_1 M_l)Q_{k,1} \nonumber \\
 \le&~  \Big(\max_{1\le k\le L_u}\dfrac{Q_{k,2}}{Q_{k,1}}\Big)\Big(\prod_{l=1}^{\Tu-1}\dfrac{b_2+c_2 M_l}{b_1+c_1 M_l}\Big)\sum_{k=1}^{\Tu} \prod_{l=1}^{k-1}(b_1+c_1 M_l)Q_{k,1}  \nonumber \\
\le&~\Big(\prod_{l=1}^{\Tu-1}\dfrac{b_2+c_2 M_l}{b_1+c_1 M_l}\Big)\Big(\max_{1\le k\le L_u}\dfrac{Q_{k,2}}{Q_{k,1}}\Big)X_{\Tu,1}  \nonumber 
 \\
\le&~\Big(\prod_{l=1}^{\Tu}\dfrac{b_2+c_2 M_l}{b_1+c_1 M_l}\Big)u^{\epsilon/\alpha_1}u^{(1+\epsilon)/\alpha_1} \nonumber \\
\le&~e^{\sum_{l=1}^{\Tu}W_{l,2}-W_{l,1}}\,u^{(1+2\epsilon)/\alpha_1}\,.  \label{eq:boundX2star}
\end{align}
Now we use that on $C_u(\varepsilon)$ we have the relation 
$$
\sum_{l=1}^{\Tu}(W_{l,2}-W_{l,1}) \le T_u(\mu_{2|1}-\mu_{1|1}) + \epsilon\Tu ~\le~ L_u (\mu_{2|1}-\mu_{1|1} + \epsilon)
$$
so that \eqref{eq:boundX2star} yields
\begin{align} \frac{\log X_{\Tu,2}}{\log u} ~\le&~ \frac{\mu_{2|1}-\mu_{1|1}+\epsilon}{\mu_{1|1} \alpha_1}+  + \frac{1+3 \epsilon}{\alpha_1} \nonumber \\
=&~ \frac{\mu_{2|1}+\epsilon(1+3 \mu_{1|1})}{\mu_{1|1}\alpha_1} ~=~ \frac{1}{\alpha_2} \frac{\alpha_2 \mu_{2|1}+ \epsilon \alpha_2(1+ 3 \mu_{1|1})}{\alpha_1 \mu_{1|1}} \label{eq:bound313}.
\end{align}
Here we have used that 
$$\exp\big(\sqrt{\log u}\big)\log u = \exp\big(\log u / \sqrt{\log u}\big)\log u = u^{1/\sqrt{\log u}}\log u \le u^{\epsilon/\alpha_1} $$ for any fixed $\epsilon>0$, as soon as $u$ is large enough.

Under the condition \eqref{eq:condition} it is always possible to find $\epsilon$ so small that
$$
\eta:=\frac{1}{\alpha_2} \frac{\alpha_2 \mu_{2|1}+ \epsilon \alpha_2(1+ 3 \mu_{1|1})}{\alpha_1 \mu_{1|1}} ~\le~ \frac{1}{\alpha_2}-\epsilon\,
$$
and hence by \eqref{eq:bound313}, 
$$
\{X_{\Tu,2} > u^{1/\alpha_2}/2 \}\cap C_u(\epsilon) \subset\big\{u^\eta \ge X_{\Tu,2} > u^{1/\alpha_2}/2\big\}= \emptyset
$$
for $u$ sufficiently large.
 It follows that the first term $A_u$ in \eqref{eq:decomposing V} does not contribute on $C_u(\varepsilon)$.
 
 \medskip
 
 \step 
 Turning to $B_u$, we start by bounding the multiplicative factor on $C_u(\epsilon)$. By Lemma \ref{lem:eqcond},
 \begin{align*}
 \prod_{l=1}^{\Tu}(b_2+c_2 M_l) ~=&~ \exp \bigg( \sum_{l=1}^{\Tu} (W_{l,2} - W_{l,1}) \bigg) \exp \bigg( \sum_{l=1}^{\Tu} W_{l,1} \bigg) \\
 \le&~ e^{ L_u (\mu_{2|1}-\mu_{1|1} + \epsilon)} \, u^{(1+\epsilon)/\alpha_1} ~\le~ u^{\eta}
 \end{align*}
where we used the same calculations as the ones leading to \eqref{eq:bound313}.
Hence 
\begin{align*}
&~\P\Big( \Big\{ \prod_{l=1}^{\Tu}(b_2+c_2 M_l) X_{\Tu:\infty,2} > \frac12 u^{1/\alpha_2}\Big\} \cap C_u(\epsilon) \,\Big| \,\Tu< \infty\Big) \\
 ~\le&~ \P\big( X_{\Tu:\infty,2}  >  u^{1/\alpha_2-\eta}/2\big| \Tu<\infty \big)= \P\big( X_2  >   u^{1/\alpha_2-\eta}/2 \big)\,.
\end{align*}
since $X_{\Tu:\infty,2}$ is independent of $\{\Tu < \infty\}$. Since $1/\alpha_2 > \eta$, the last probability tends to zero. 

\medskip

Combining the two previous steps, we have proved that
$$\lim_{u\to \infty}\P \bigg( \Big\{X_2>u^{1/\alpha_2} \Big\} \cap  C_u(\epsilon)  \, \bigg| \, {\Tu} < \infty \Big)
= 0
$$ 
which by Lemma \ref{lem:eqcond} is enough to conclude \eqref{eq:condtu} and thus the desired result.
\end{proof}

\section{The diagonal SRE with equal coefficients}\label{sec:diageq}
In  this section we focus on the case where $b_i=b\ge 0$ and $c_i=c>0$ for any $1\le i\le d$ so that 
\[
\bb X_t=(b+cM_t) \bb X_{t-1}+\bb Q_t,\qquad t\in \Z.
\]
We can interpret the multiplicative factor $(b+cM_t)$ as multiplication with the random similarity matrix  $(b+cM_t) I_d$, thus we are in the framework of \cite{buraczewskietal}. 
%
%
%
From there, we obtain the following result:

\begin{thm}\label{thm:mvar equal components}
Assume \eqref{eq:A1}--\eqref{eq:A5} for all $1 \le i \le d$. Let $\bb X_0$ have the stationary distribution. Then $\bb X_0$ is VSRV and $({\bb X}_t)_{t \ge 0}$ is a VSRV process of order $\bb \alpha=(\alpha, \dots, \alpha)$, and its spectral tail process satisfies the relation
$$\widetilde{\bb  \Theta}_t=(b+cM_{t}) \widetilde{\bb \Theta}_{t-1}, \quad t \ge 1.$$
\end{thm}

\begin{proof}
By \cite[Theorem 1.6]{buraczewskietal}, there is a non-null Radon measure $\mu$ on $[-\infty, \infty]^d \setminus \{0\}$ such that
$$  x^{\alpha} \P(x^{-1} \bb X_0 \in \cdot) ~\stackrel{v}{\to}~ \mu, \qquad x \to \infty.$$ [See \cite[Theorem 4.4.21]{Buraczewski.etal:2016} for a reformulation of the quoted result which is more consistent with our notation.] Hence, $\bb X_0$ is (standard) regularly varying and also VSRV of order $\bb \alpha=(\alpha, \dots, \alpha)$ since ${\bb\Theta}_0$ and $\widetilde{\bb \Theta}_0$ coincide then. 

The remaining assertions follow from a direct application of Proposition \ref{prop:SREVSRV}.
\end{proof}

%

In order to determine whether the components of $\bb X_0$ are asymptotically independent or dependent, we are interested in information about the support of $\P(\widetilde{\bb \Theta}_0 \in \cdot)$.
We write $\supp(\bb Q)$ for the support of the law of $\bb Q$ and $\mathrm{span}(E)$ for the linear space spanned by set $E \subset \R^d$. Let $S^{d-1}_\infty$ denote the unit sphere in $\R^d$ with respect to $\anorm{\cdot}$ which coincides with the unit sphere for the max-norm whatever is $\bb \alpha$. 

\begin{lem}\label{lem:support spectral measure}
	Under the assumptions of Theorem \ref{thm:mvar equal components},
	\begin{equation}  \label{eq:suppTheta} \mathrm{supp}( \widetilde{\bb \Theta}_0) ~\subset~\mathrm{span}\big(\supp(\bb Q)\big) \cap S^{d-1}_\infty.\end{equation}
In addition, the following implications hold:
\begin{enumerate}[(a)]
\item \label{lem52a} If $b=0$, $c>0$ and $\mathrm{supp}(M)$  is dense in $\R $, then
$$  \mathrm{supp}( \widetilde{\bb \Theta}_0) ~=~\mathrm{span}\big(\supp(\bb Q)\big) \cap S^{d-1}_\infty.$$
\item \label{lem52b} If  $b>0$, $c>0$ and $\mathrm{supp}(M)$  is dense in $\R_+$, then
$$ \mathrm{supp}( \widetilde{\bb \Theta}_0) ~=~ \{ a_1 \bb q_1 + \dots + a_n \bb q_n \, : \, n \in \N,  a_i > 0, \bb q_i \in \mathrm{supp}(\bb Q) \} \cap S^{d-1}_\infty,$$
{\em i.e.} it equals the convex cone generated by $\mathrm{supp}(\bb Q)$ intersected by the unit sphere.
\item If $\mathrm{supp}(\bb Q)$ is dense in $\R^d$,  then  $\mathrm{supp}(\widetilde{\bb \Theta}_0)   ~=~S^{d-1}_\infty.$
\end{enumerate}
\end{lem}


\begin{proof}[Proof of Lemma  \ref{lem:support spectral measure}]

The proof  is based on \cite[Remark 1.9]{buraczewskietal}, which gives that the support of the spectral measure $\sigma_\infty$ with respect to the Euclidean norm is given by the directions (subsets of the unit sphere $S^{d-1}$) in which the support of $\bb X_0$ is unbounded. More precisely, consider the measures 
$$ \sigma_t(A)~:=~ \P\Big(\|\bb X_0\|_2>t, \frac{\bb X_0}{\|\bb X_0\|_2} \in A \Big)$$
Then $\mathrm{supp}(\sigma_\infty) = \bigcap_{t >0} \mathrm{supp} (\sigma_t)$. The surprising part of this result is that all directions, in which the support of $\bb X_0$ is unbounded, do matter. One does not need a lower bound on the decay of mass at infinity. But if we know that the support of the spectral measure w.r.t.  the Euclidean norm is the intersection of a particular subspace with the unit sphere, we immediately deduce the same for the spectral measure w.r.t the max-norm, {\em i.e.,} for $\P( {\bb \Theta}_0 \in \cdot)$, as well as for $\P( \widetilde{\bb  \Theta} \in \cdot)$.

Thus, to proceed, we have to study the support of $\bb X_0$. For simplicity we  write, for the remainder of the proof, $(m, \bb q)$ for a realization of the random variables $(b+cM, \bb Q)$.
 We identify a pair $(m, \bb q)$ with the affine mapping $h(\bb x)=m\bb x+\bb q$, we say that $h \in \mathrm{supp}\big((b+cM,\bb Q)\big)$ if $(m,\bb q) \in \mathrm{supp} \big((b+cM, \bb Q)\big)$. We consider the semigroup generated by mappings in $\mathrm{supp}\big((b+cM, \bb Q)\big)$,
$$ \mathcal{G} ~:=~ \Big\{ h_1 \cdots h_n \, : \, h_i \in \mathrm{supp}\big((M, \bb Q)\big), \ 1 \le i \le n, \ n \ge 1 \Big\}.$$
Then, by \cite[Lemma 2.7]{buraczewskietal}
$$ \mathrm{supp}\big(\bb X_0\big) ~=~ \text{closure of } \Big\{ \tfrac{1}{1-m} \bb q \, : \, (m, \bb q) \in \mathcal{G}, \, |m|<1 \Big\}.$$
This is obviously a subset of $\mathrm{span}({\bb Q})$, hence \eqref{eq:suppTheta} holds.
[Again, see \cite[Proposition 4.3.1]{Buraczewski.etal:2016} for a reformulation of the quoted result which is more consistent with our notation.]

\medskip

Since $M$ and $\bb Q$ are independent, $\mathrm{supp}\big((b+cM,\bb Q)\big)=\mathrm{supp}(b+cM) \times \mathrm{supp}(\bb Q)$ and a general element in $\mathcal{G}$ is of the form
$$ h(\bb x)=m_1 \cdots m_n \bb x + \Big( \bb q_1 +\sum_{k=2}^n m_1 \cdots m_{k-1} \bb q_k \Big)$$
with $ m_i \in \mathrm{supp}(b+cM)$, $\bb q_i \in \mathrm{supp}(\bb Q)$. 
Thus, a generic point in $\mathrm{supp}(\bb X_0)$ is of the form
\begin{equation}\label{eq:generic element in supp S} \frac{1}{1-m_1 \cdots m_n} \Big( \bb q_1 +\sum_{k=2}^n m_1 \cdots m_{k-1}  \bb q_k \Big),\end{equation}
with 
$$ m_i \in \mathrm{supp}(b+cM), \ \bb q_i \in \mathrm{supp}(\bb Q), \ |m_1 \cdots m_n| <1.$$
The prefactor in \eqref{eq:generic element in supp S} is scalar, while the  bracket term represents a linear combination of $q_k \in \mathrm{supp}(\bb Q)$. Now we can prove the two implications.

Concerning (a), if $\mathrm{supp}(M)$ is dense in $\R$, then the bracket term in  \eqref{eq:generic element in supp S} can be chosen such that its direction approximates any direction of  $y \in \mathrm{span}\big(\mathrm{supp}(\bb Q)\big)$. Then, given $t>0$, $m_n$ can be chosen arbitrarily small, such that $|m_1 \dots m_n| <1$ and moreover, the norm of \eqref{eq:generic element in supp S} exceeds $t$. It follows that $\mathrm{supp}(\sigma_t) = \mathrm{span}\big(\mathrm{supp}(\bb Q)\big) \cap S^{d-1}$ for all $t$, which yields the assertion since $\mathrm{supp}(\sigma_\infty)=\bigcap_{t >0} \mathrm{supp}(\sigma_t)$.

Concerning (b), note that $m \in \mathrm{supp}(b+cM)$ is bounded from below by $b>0$, with $b<1$ due to assumption \eqref{eq:A1}. If $\mathrm{supp}(M)$ is dense in $\R_+$, then the bracket term in  \eqref{eq:generic element in supp S} can be chosen such that its direction approximates any direction of  $y \in \mathrm{span}\big(\mathrm{supp}(\bb Q)\big)_+$ given that its norm is suitably large. Similarly as for (a) above, the desired assertion follows. 

Concerning (c), if $\mathrm{supp}(\bb Q)$ is dense in $\R^d$, then the bracket term can be chosen such that it approximates an arbitrary element of $\R^d$ and its modulus is larger than $t$, while \eqref{eq:A1} entails that there are $m_i \in \mathrm{supp}(M)$ such that $|m_1 \dots m_n|<1$.
\end{proof}
%

%
%
%

\section{Diagonal SRE  and multivariate GARCH models}\label{sec:generalmodel}

\subsection{The diagonal SRE - the general case}\label{subsection61}
In this section we study the vector scaling regular variation properties of the diagonal SRE in full generality. We suppose that coordinates are chosen in such a way that   $\alpha_i$ decreases with $i$. We partition  $\{1, \dots, d\} = I_1  \cup I_2 \cup \dots \cup I_r$ such that $b_i=b_j$, $c_i=c_j$ and then $\alpha_i=\alpha_j$ if and only if $i,j \in I_\ell$ for some $1 \le \ell \le  r$.
We further denote by
$$ \R^{|I_\ell|} ~=~\{x\in \R^d; \ x_i=0\;\mbox{for}\; i\notin I_{\ell}\}$$
the (embedded) subspace corresponding with coordinates indexed by $I_{\ell}$ and by 
$$S_\infty^{|I_{\ell}|-1}~=~\{x\in \R^d; \max_{i \in I_{\ell}}|x_i|=1 \; \mbox{and}\; x_i=0\;\mbox{for}\; i\notin I_{\ell}\}$$
its max-norm-unit sphere. Note that if $I_{\ell}=\{i\}$ is a singleton, then $S_\infty^{|I_{\ell}|-1}=\{\bb e_i, - \bb e_i\}.$ We extend  \eqref{case1'} and  \eqref{case2'} so that for any $\ell\neq\ell'$ such that $\alpha_i>\alpha_j$ for $i\in I_\ell$ and $j\in I_{\ell'}$ we either have
\begin{align}
\tag{Case I'}   &b_i = b_j = 0, \qquad  c_j >c_i >0, \qquad & &\text{$M_t$ is $\R$-valued}  \\
\tag{Case II'}   &b_j \ge b_i >0, \qquad c_j > c_i >0, \qquad \frac{c_j}{c_i} \ge \frac{b_j}{b_i}, \qquad & &\text{$M_t >0$ a.s.}
\end{align}

\begin{thm} \label{thm:general case} Let $(\bb X_t)$ a stationary process satisfying the diagonal SRE \eqref{eq:SRE} and assume that \eqref{eq:A1}--\eqref{eq:A6} hold. Assume that \eqref{case1'} or  \eqref{case2'} holds,
then $(\bb X_t)$ is a VSRV process satisfying
\begin{equation} \label{eq:support Theta}
 \mathrm{Supp}(\widetilde{\bb \Theta}_0) \ \subset \ \cup_{1\le \ell \le r}S_\infty^{|I_\ell|-1}
\end{equation}
and 
\begin{equation}\label{eq:spectral process thm51}
\widetilde{\bb \Theta}_t = \mathbb{M}_t   \widetilde{\bb\Theta}_{t-1}\,,\qquad t\ge 1\,.
\end{equation}
\end{thm}

\begin{proof}	

We start by proving that $(\bb X_t)$ is a VSRV process. According to Proposition \ref{prop:SREVSRV} and Remark \ref{rem:SREVSRV}, it suffices to prove that $\bb X_0$ is VSRV, then \eqref{eq:spectral process thm51} and the VSRV of $(\bb X_t)$ follow.

We use the following short-hand notation: For $\bb x \in \R^d$, let $\bb x_{\ell}=(x_i)_{i \in I_{\ell}}$, $\|\bb x\|_{\ell} := \max_{i \in I_{\ell}} |x_i| $  and  $\alpha(\ell)$ is the common tail index of all coordinates in $I_{\ell}$.	
	
Let $\epsilon >0$, $\ell \neq k$. By Eq. \eqref{eq:asympIndep} of Theorem \ref{thm:asymptotic independence}, it holds that
\begin{align}
&~\lim_{x \to \infty} x \cdot\P \Big( \| \bb X_0\|_\ell > \epsilon \anorm{\bb X_0}^{1/\alpha(\ell)}, \ \| \bb X_0\|_k > \epsilon \anorm{\bb X_0}^{1/\alpha(k)}, \ \anorm{\bb X_0} >x\Big) \nonumber \\
\le&~\lim_{x \to \infty} x \cdot\P \Big( \| \bb X_0\|_\ell > \epsilon x^{1/\alpha(\ell)}, \ \| \bb X_0\|_k > \epsilon x^{1/\alpha(k)}\Big) \nonumber \\
\le&~  \sum_{i \in I_{\ell},\, j \in I_k} \lim_{x \to \infty} x \cdot \P\Big( |X_{0,i}| > \epsilon x^{1/\alpha_i}, \ |X_{0,j}| > \epsilon x^{1/\alpha_j}\Big) ~=~ 0 \label{eq:products vanish}
\end{align}	
We note from the results of Section \ref{sec:diageq}  that there are positive constants $c_{\ell}$ and probability measures $\xi_{\ell}$ on the $|I_\ell|$-dimensional unit sphere (w.r.t. the max-norm), such that for all $1 \le \ell \le r$
\begin{equation*}
\lim_{x \to \infty} x \cdot \P\Big( \|\bb X_0\|_{\ell} > x^{1/\alpha(\ell)} , \ \|\bb X_0\|_{\ell}^{-1} \bb X_{0, \ell} \in \cdot \Big) ~=~ c_{\ell}\, \xi_\ell(\cdot) . 
\end{equation*} 
Applying the inclusion-exclusion principle, we have
\begin{align}
&~ \lim_{x \to \infty} x \cdot \P\big( \anorm{\bb X_0}>x\big) ~=~ \lim_{x \to \infty} x \cdot \P \Big( \bigvee_{1 \le \ell \le r} \|\bb X_0\|_{\ell} > x^{1/\alpha(\ell)}\Big) \nonumber \\
=&~ \sum_{1 \le \ell \le r} \lim_{x \to \infty} x \cdot \P \Big(  \|\bb X_0\|_{\ell}  > x^{1/\alpha(\ell)}\Big) \nonumber 
\\
&~\quad - \sum_{1 \le \ell < k \le r} \lim_{x \to \infty} x \cdot \P\Big(  \|\bb X_0\|_{\ell}  > x^{1/\alpha(\ell)}, \  \|\bb X_0\|_{k}  > x^{1/\alpha(k)}\Big) + \dots \nonumber \\
=&~c_1 + \dots + c_r ~=:~c, \label{eq:limitxoregvar}
\end{align}
since all intersection terms vanish asymptotically due to Eq. \ref{eq:products vanish} (with $\epsilon=1$).

Thus we have shown that $\anorm{\bb X_0}$ is regularly varying.
We claim that 
\begin{equation*}
\lim_{x \to \infty} \P \Big( \anorm{\bb X_0}^{-1/\bb \alpha} \bb X_0 \in \cdot \, \Big| \, \anorm{\bb X_0} > x \Big) ~=~ \frac{1}{c}\sum_{1 \le \ell \le r} c_\ell \, \widetilde{\xi}_{\ell}(\cdot),
\end{equation*}
where 
$$ \widetilde{\xi}_{\ell} ~=~ \delta_{\bb 0_1} \otimes \dots \otimes \delta_{\bb 0_{\ell -1}} \otimes \xi_\ell \otimes \delta_{\bb 0_{\ell+1}} \otimes \dots \otimes \delta_{\bb 0_r}$$
is the extension of $\xi_{\ell}$ to a measure on the unit sphere $S^{d-1}_\infty$ in $\R^d$  by putting unit mass in the origin of the additional coordinates. Hence, its support is contained in $S_\infty^{|I_\ell|-1}$. In particular, \eqref{eq:support Theta} follows once this claim is proved.

By the Portmanteau lemma, it suffices to study closed sets. Note that for any closed set $B \subset S^{d-1}_\infty$, it holds that
$$  B_{\ell, \epsilon} ~:=~\{ \bb x_\ell \, : \,  \bb x \in B, \, |x_j|<\epsilon \text{ for } j \notin I_\ell \}~\to~ \{ \bb x_\ell \, : \,  \bb x \in B \cap S^{|I_\ell|-1} \} ~=:~ B_\ell$$
as $\epsilon \to 0$. Using \eqref{eq:products vanish} and the inclusion-exclusion principle, we obtain
\begin{align*}
&~\limsup_{x \to \infty}  x \cdot \P\Big(\anorm{\bb X_0}^{-1/\bb \alpha} \bb X_0 \in B, \, \anorm{\bb X_0} > x \Big) \\
~=&~ \limsup_{x \to \infty}  x \cdot \P\Big(\anorm{\bb X_0}^{-1/\bb \alpha} \bb X_0 \in B, \,  \bigvee_{1 \le k \le r} \|\bb X_0\|_{k} > x^{1/\alpha(k)} \Big) \\
~=&~ \sum_{1 \le \ell \le r} \limsup_{x \to \infty}  x \cdot \P\Big(\anorm{\bb X_0}^{-1/\bb \alpha} \bb X_0 \in B, \|\bb X_0\|_\ell > x^{1/\alpha(\ell)}, \,
\\ & \hspace*{5.5cm}  \bigwedge_{k \neq \ell} \| \bb X_0\|_k \le \epsilon \anorm{\bb X_0}^{1/\alpha(k)}\Big) \\
~\le&~ \sum_{1 \le \ell \le r} \lim_{x \to \infty}  x \cdot \P\Big(\|\bb X_0\|_\ell^{-1} \bb X_{0,\ell} \in B_{\ell,\epsilon},\, \|\bb X_0\|_\ell > x^{1/\alpha(\ell)}, \, \\
& \hspace*{6cm} \bigwedge_{k \neq \ell} \| \bb X_0\|_k \le \epsilon x^{1/\alpha(k)} \Big) \\
~=&~ \sum_{1 \le \ell \le r} \lim_{x \to \infty}  x \cdot \P\Big(\|\bb X_0\|_\ell^{-1} \bb X_{0,\ell} \in B_{\ell,\epsilon}, \, \|\bb X_0\|_\ell > x^{1/\alpha(\ell)} \Big) \\
~=&~ \sum_{1 \le \ell \le r} c_l\, \xi_\ell({B_\ell,\epsilon})
\end{align*}
This holds for all $\epsilon >0$. Since the sequence $B_{\ell,\epsilon}$ is decreasing, we conclude by the continuity of $\xi_\ell$ that
\begin{align*} \limsup_{x \to \infty}   x \cdot \P\Big(\anorm{\bb X_0}^{-1/\bb \alpha} \bb X_0 \in B, \, \anorm{\bb X_0} > x \Big) ~\le&~ \sum_{1 \le \ell \le r} c_\ell \, \xi_{\ell} (B_\ell) \\~=&~ \sum_{1 \le \ell \le r} c_\ell \widetilde{\xi}_\ell(B). \end{align*}
Combined with \eqref{eq:limitxoregvar}, this proves the weak convergence by an application of the Portmanteau lemma.
\end{proof}

\subsection{Diagonal BEKK-ARCH($1$) model}
We consider $(\bb X_t)$ the solution of the diagonal BEKK-ARCH(1) model defined as in \cite{pedersen:wintenberger:2018} by the system
\[
\begin{cases}\bb X_t&=\bb H_t^{1/2}\bb Z_t,\qquad t\in \Z,\\
\bb H_t&= \bb \Sigma+{\rm Diag}(c_1,\ldots,c_d)\bb X_{t-1}\bb X_{t-1}^\top{\rm Diag}(c_1,\ldots,c_d)\,,
\end{cases}
\]
where $(\bb Z_t)$ is an iid sequence of Gaussian random vectors $\mathcal N_d(0, I)$ and $ \bb \Sigma$ is a variance matrix.  Then $\bb X_t$ satisfies the diagonal SRE model \eqref{eq:SRE} with  $(M_t)$  iid $\mathcal{N}(0,1)$, $b_i=0$ for all $1\le i\le d$ and $(\bb Q_t)$ are iid  $\mathcal{N}(0,\bb{\Sigma})$. Under the top-Lyapunov condition
\begin{equation}\label{eq:stationarityassumption}
c_i^2<2 e^\gamma, \qquad 1\le i\le d,
\end{equation}
 where $\gamma \approx 0.5772$ is the Euler constant, it exists a stationary solution $(\bb X_t)$;  see e.g. \cite{Nelson1990StationarityGARCH}.  Its multivariate extremal behaviour is given by the following corollary
which follows from an application Theorem \ref{thm:general case} in \eqref{case1'}:

\begin{cor}\label{cor:resultBEKK}
If the stationarity assumption \eqref{eq:stationarityassumption} is satisfied, then the stationary solution $(\bb X_t)$ of the diagonal BEKK-ARCH(1) model is a VSRV process satisfying
\begin{equation*} 
\mathrm{Supp}(\widetilde{\bb \Theta}_0) \ = \ \cup_{1\le \ell \le r}S_\infty^{|I_\ell|-1}
\end{equation*}
and 
\begin{equation}
\widetilde{\bb \Theta}_t = M_t \, \diag(c_1,\ldots, c_d) \widetilde{\bb\Theta}_{t-1}\,,\qquad t\ge 1\,.
\end{equation} 
\end{cor}

See the first paragraph in Subsection \ref{subsection61} for the definition of $I_{\ell}$; here $c_i=c_j$ if and only if there exists an $\ell$ such that $i,\,j\in I_\ell$.
\begin{proof}
We have to check the assumptions of the previous theorem. This is readily done for \eqref{eq:A1}-\eqref{eq:A5}, see \cite{pedersen:wintenberger:2018} for details. Considering \eqref{eq:A6}, let $\sigma_i^2=\mathrm{Var}(Q_i)$ and $\rho_{ij}$ be the correlation coefficient of $Q_i$ and $Q_j$; $\E Q_i = \E Q_j=0$. Then the ratio $Q_i/Q_j$ has a Cauchy distribution with location parameter $a=\rho_{ij} \tfrac{\sigma_i}{\sigma_j}$ and scale parameter $b=\tfrac{\sigma_i}{\sigma_j} \sqrt{1-\rho_{ij}^2}$; see e.g. \cite[Eq. (3.3)]{Curtiss:1941}. The Cauchy distributions are 1-stable, hence
$$ \P\Big( \frac{|Q_i|}{|Q_j|} > u \Big) = O(u)$$
and \eqref{eq:A6} follows if $I$, $J$ are singletons. To compare $Q_{I}^*=\max_{i \in I} |Q_{i}|$ with $Q_{J}^*=\max_{j \in J} |Q_j|$
we use the simple bound (fix any $j \in J$)
$$ \Big\{ \frac{Q_{I}^*}{Q_{J}^*} > u  \Big\} ~\subset~ \bigcup_{i \in I} \Big\{ \frac{|Q_i|}{|Q_k|} >u \Big\}  $$
to conclude that the probability of this event still decays as $O(u)$. Thus \eqref{eq:A6} also holds in this case.
%
%

It remains to show that $\mathrm{supp}\big(\widetilde{\bb \Theta}_0\big)$ is {\em equal} to $\cup_{1\le \ell \le r}S^{|I_\ell|-1}$. Therefore, we can focus on a particular block $I$ and show that the spectral measure of the restriction $(X_{0,i})_{i\in I}$ has full support $S^{|I|-1}$. 

If $I$ is a singleton, then this means nothing but that left and right tails are regularly varying with the same index; which already follows from the Goldie-Kesten theorem, see \eqref{eq:tails}. 
If $|I|>1$ then we are in the setting of Section \ref{sec:diageq}. The result follows from Lemma \ref{lem:support spectral measure} \eqref{lem52a}, since $M$ and $(Q_i)_{i \in I}$ are independent Gaussians, and $\mathrm{span}\big(\mathrm{supp}((Q_i)_{i \in I})\big)=\R^{|I|}$ since $C$, the variance of $\bb Q$, has full rank.
\end{proof}

The multivariate regular variation properties of the diagonal BEKK-ARCH(1) process is quite simple as the support is preserved by the multiplicative form of the tail process: The tail process is a mixture of multiplicative random walks with distinct supports. Each support corresponds to the span of the diagonal coefficients of the multiplicative matrix that are equal. From a risk analysis point of view, it means that the extremal risks are dependent and of similar intensity  only in the directions of equal diagonal coefficients. Our multivariate analysis appeals for an extreme financial risk analysis based on the estimation of the diagonal coefficients of the BEKK-ARCH(1) process accompanied with a test of their equality. However by asymptotic independence the mixture is asynchronous and the model is relevant only if the each group of dependent extremal risks are due to different financial crisis.

\medskip

Under more restrictive assumptions, it is also possible to provide second-order results, see the arXiv version of the present paper \cite{Mentemeier.Wintenberger:2019}. See also \cite{matsui2019characterization} for a recent analysis of the general BEKK-ARCH model.

\subsection{CCC-GARCH($1,1$) model}

The Constant Conditional Correlation CCC-GARCH(1,1) model has been introduced by \cite{Bollerslev1990} such as the stationary solution of  the system 
\begin{equation}\label{eq:ccc}
\begin{cases}
\bb R_t&=\bb \Sigma_t N_t\,,\qquad t\in \Z\,,\\
\bb \Sigma_{t}&=\text{Diag}(\sigma_{t,1}^2,\ldots,\sigma_{t,d}^2)\,, \\
\sigma_{t,i}^2& = a_i+b_{i}\sigma_{t-1,i}^2+c_{i}R_{t-1,i}^2 \,,
\end{cases}
\end{equation}
where $N_t$  iid is distributed as $\mathcal N_d(0, \bb \Sigma)$ for $\bb \Sigma$ a correlation matrix and the coefficients  $a_i$s, $b_i$s and $c_i$s are positive.  The general CCC-GARCH(1,1) model of \cite{starica:1999} is defined by the same system with extra cross terms $\sum_{j\neq i}b_{i,j}\sigma_{t-1,j}^2$ and $\sum_{j\neq i}c_{i,j}R_{t-1,j}^2$ in the second equation. By a direct  application of Theorem of \cite{Kesten1973}, the volatility $\bb X_t=(\sigma_{t,1}^2,\ldots,\sigma_{t,d}^2)^\top\in\R^d_+$ of a  general CCC-GARCH(1,1) is regularly varying when $b_{i,j}+c_{i,j}>0$ for all $i,j$, see  \cite{starica:1999} for more details.

Back to the original CCC-GARCH model \eqref{eq:ccc} we consider the degenerate case $N_t=(1,\ldots,1)^\top Z_t$ with $Z_t\sim \mathcal N(0,1)$. The existence of a stationary solution is ensured under the top-Lyapunov conditions
\begin{equation}\label{eq:statccc}
\E[\log(b_i+c_iZ_0^2)]<0\,,\qquad 1\le i\le d\,.
\end{equation}
Then the original CCC-GARCH(1,1) volatility $\bb X_t=(\sigma_{t,i}^2)_{1\le i\le d}$ satisfies the diagonal SRE \eqref{eq:SRE} with $\bb Q_t\equiv \bb q =(a_1,\ldots,a_d)^\top$ and 
$$
\mathbb{M}_t=\mathrm{Diag}(b_1+c_1 Z^2_t, \dots, b_d+c_d Z^2_t)\,.
$$

We get  from another application Theorem \ref{thm:general case} in \eqref{case2'}:

\begin{cor}\label{cor:resultCCC}
If the stationarity assumption \eqref{eq:statccc} is satisfied and $c_j/c_i\ge b_j/b_i$ whenever $\alpha_i>\alpha_j$, then the volatility $(\bb X_t)$ of the stationary solution of the CCC-GARCH model \eqref{eq:ccc} with $N_t=(1,\ldots,1)^\top Z_t$ is a VSRV process satisfying
\begin{equation*} 
\mathrm{Supp}(\widetilde{\bb \Theta}_0) \ = \ \cup_{1\le \ell \le r} \{{\bb q}_\ell\}
\end{equation*}
where $\bb q_{\ell}$ is given by 
$$ (\bb q_{\ell})_i ~=~ \begin{cases}
0 & i \notin I_{\ell}, \\
\bb q_i / \bb q_{\ell}^* & i \in I_{\ell} 
\end{cases} $$
with $\bb q_{\ell}^* := \max_{j \in I_{\ell}} \bb q_j$.
Moreover 
\begin{equation*}
\widetilde{\bb \Theta}_t = \diag(b_1+c_1 Z_t^2,\ldots,b_d+ c_d Z_t^2) \widetilde{\bb\Theta}_{t-1}\,,\qquad t\ge 1\,.
\end{equation*} 
\end{cor}


\begin{proof}
	To apply Theorem \ref{thm:general case} in \eqref{case2'}, we have to check conditions \eqref{eq:A1}--\eqref{eq:A6} for $1 \le i \le d$. The stationarity assumption \eqref{eq:statccc} is \eqref{eq:A1} and implies moreover, together with the fact that $M=Z_0^2$ has a $\chi^2$-distribution, that all moments of $b_i+c_iM$ exist and $\lim_{s \to \infty} \E\big[ |b_i+c_iM|^s\big]=\infty$, while $\E\big[ |b_i+c_iM|^\epsilon\big]<1$ for sufficiently small $\epsilon >0$. Hence, for any $1\le i\le d$ there exists $\alpha_i>0$ satisfying
	$$
	\E[(b_i+c_iZ_0^2)^{\alpha_i}]=1\,.
	$$
	The further assumptions are readily checked using that $b_i+c_iM$ has a density on $[b,\infty)$ and that $\bb Q_t \equiv \bb q$ is deterministic. We conclude that \begin{equation*}
	\mathrm{supp}(\widetilde{\bb \Theta}_0) \ \subset \ \cup_{1\le \ell \le r}S_\infty^{|I_\ell|-1}\end{equation*}
	
	Finally, by Lemma \ref{lem:support spectral measure} \eqref{lem52b}, for each $I_{\ell}$-block,
	\begin{equation*}  \mathrm{supp}( \widetilde{\bb \Theta}_0) \cap S_\infty^{|I_\ell|-1}  ~=~ \{ {\bb q}_{\ell}\}.\end{equation*}
\end{proof}

Thus, despite the constant correlation matrix $\bb \Sigma$ being totally correlated, the volatility process exhibits asymptotic independence. On the other hand, if the components of the CCC-GARCH(1,1) model are independent, which is the case when $\bb \Sigma = \diag(1, \dots, 1)$, they are asymptotically independent as well.

  We conjecture asymptotic independence to be true for all choices of $N_t$ in the original CCC-GARCH(1,1) model, since the one we considered is the most dependent one.

\subsection{A simulation study}
We provide some empirical evidences on the results of Theorem \ref{thm:asymptotic independence} and Lemma \ref{lem:support spectral measure}. We simulate a trajectory of length $10^8$ of the bidimensional diagonal SRE
\begin{equation*}
\bb X_t ~=~  \vect{X_{t,1}}{X_{t,2}} ~=~ \left(\begin{matrix}b_1+c_1 M_t & 0\\0& b_2+c_2 M_t\end{matrix}\right) \vect{X_{t-1,1}}{X_{t-1,2}}  + \bb Q_t.
\end{equation*}
 We fix the threshold $x$ as the empirical ($1-10^{-5}$)-percentile of the simulated values $\big(\|\bb x_t\|_{\bb \alpha}\big)$. 

Following \eqref{eq:yt}, we consider the ratios $\big(\|\bb x\|_\alpha^{-1/\bb \alpha}\bb x_t\big)$ of the exceedances satisfying $\|\bb x_t\|_{\bb \alpha}>x$ as an approximation of the spectral component $\widetilde{ \bb \Theta}_0$. We estimate empirically the angular measure $\arctan(\tilde \Theta_{0,1}/\tilde \Theta_{0,2})$ by an histogram. 

\begin{figure}[h!]
\begin{minipage}[b]{.5\linewidth}
\centering \includegraphics[height=0.25\textheight]{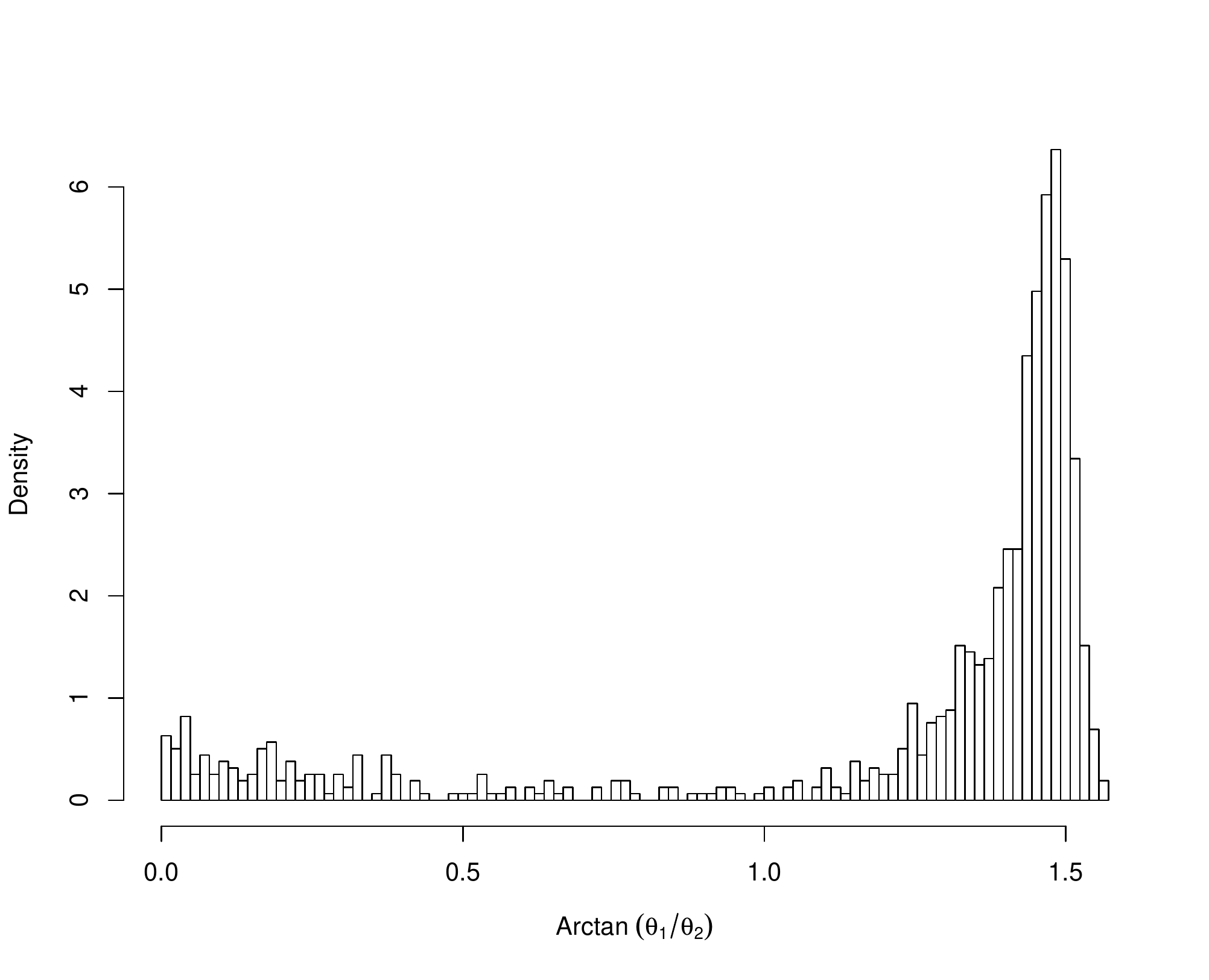}
\caption{\footnotesize  $\alpha_1=2$ and $\alpha_2=4$}
\label{fig:c1different}
\end{minipage} \hfill
\begin{minipage}[b]{.5\linewidth}
\centering    \includegraphics[height=0.25\textheight]{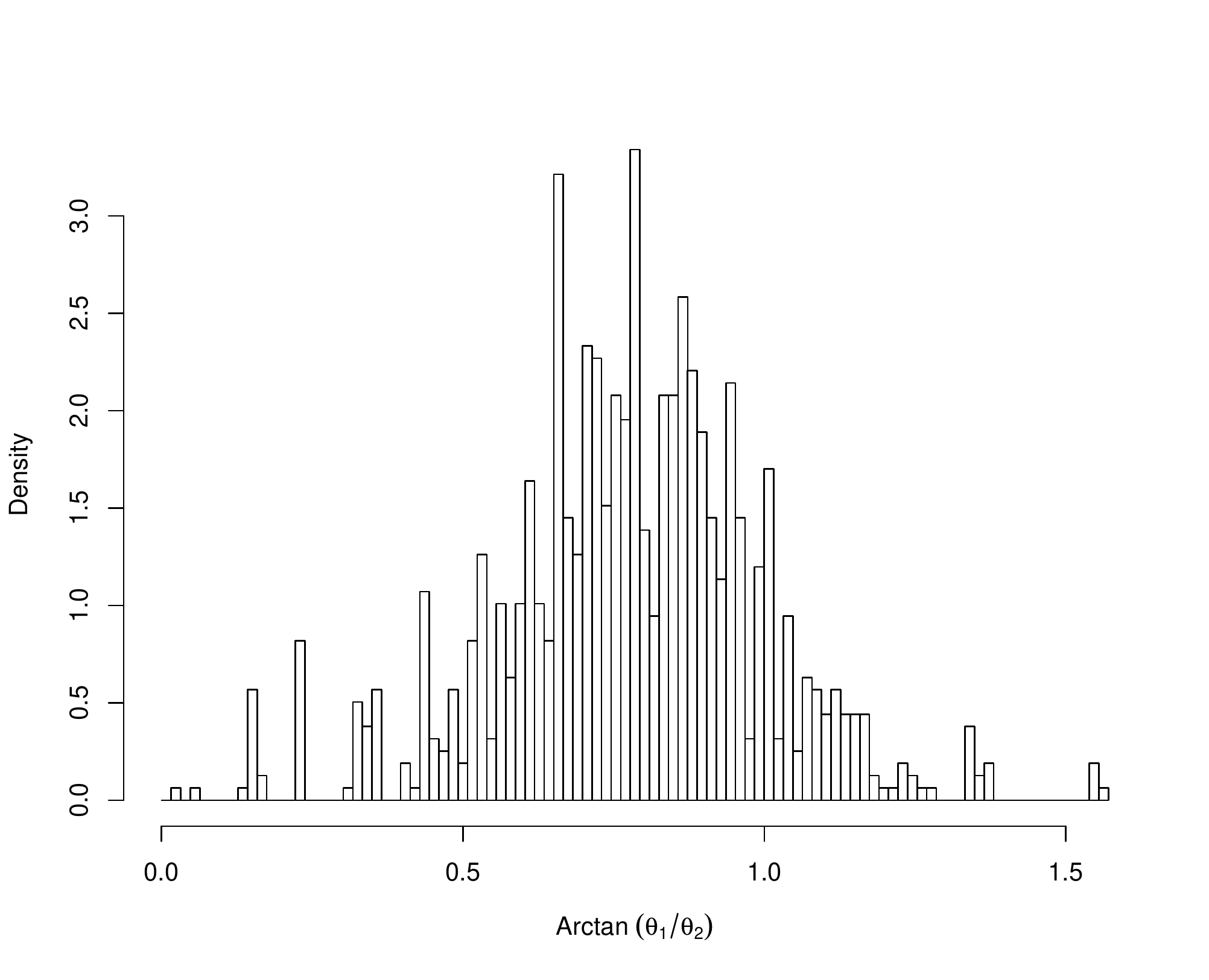}
\caption{\footnotesize  $\alpha_1=\alpha_2=2$}
\label{fig:c1same}
\end{minipage}
\end{figure} 

We first consider \ref{case1'} with $M_t\sim \mathcal N(0,1)$ and $\bb Q_t$ a bivariate standard gaussian vector with correlation $0.9$. This corresponds to the diagonal BEKK-ARCH modell. To simplify the graphical presentation, we consider the spectral measure for the absolute values of $\bb X_t$.  Figure \ref{fig:c1different} corresponds to different  coefficients $c_1=1$ and $c_2=(1/3)^{1/4}$ whereas  Figure \ref{fig:c1same} corresponds to equal coefficients $c_1=c_2=1$. In accordance with Theorem \ref{thm:asymptotic independence}, the angular measure in Figure \ref{fig:c1different} is concentrated around $0$ and $\pi/2$. Following Lemma \ref{lem:support spectral measure} (c) the support of the angular measure should be $[0,\pi/2)$ which is not in contradiction with Figure \ref{fig:c1same}.

\begin{figure}[h!]
\begin{minipage}[b]{.5\linewidth}
\centering \includegraphics[height=0.25\textheight]{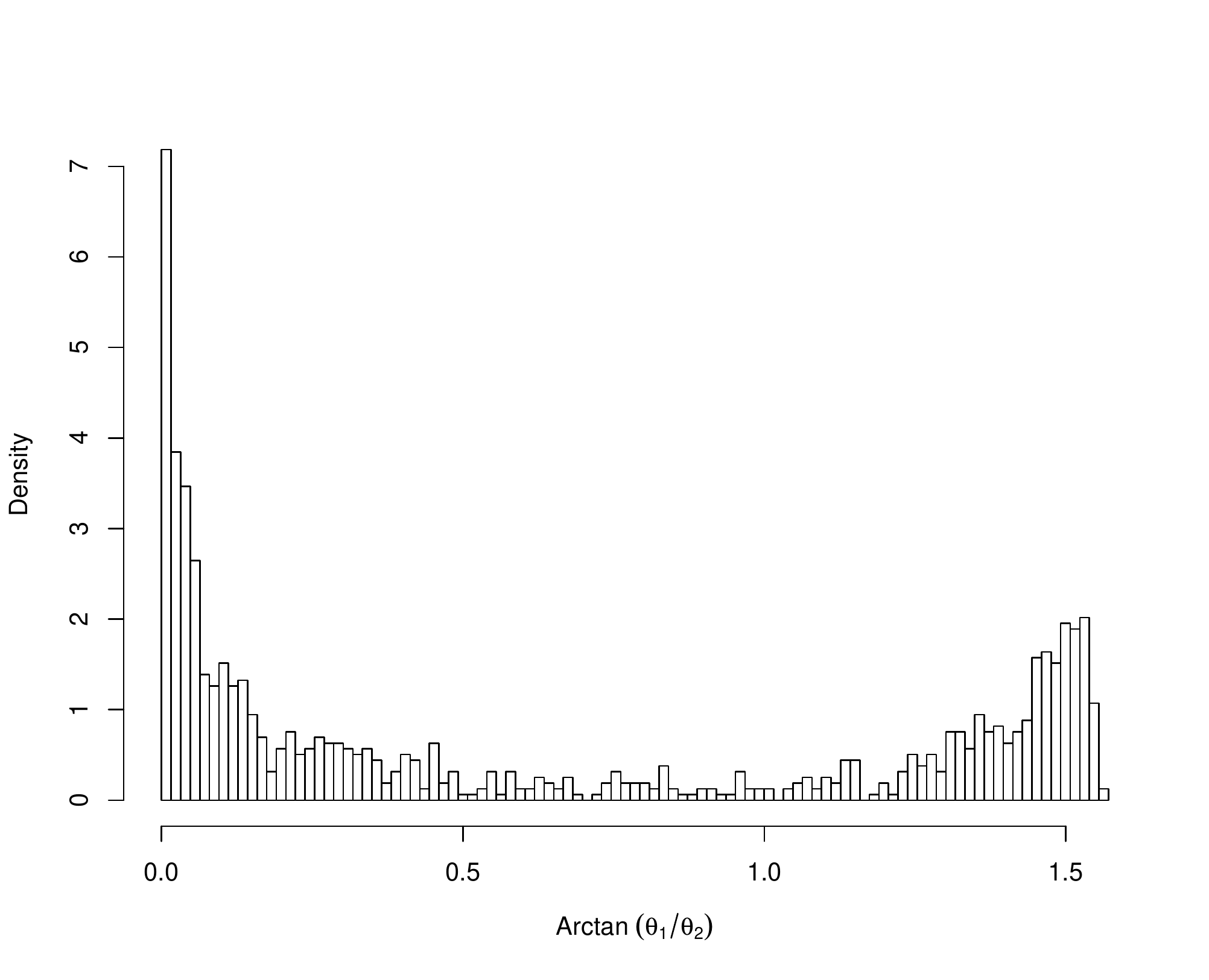}
\caption{\footnotesize  $\alpha_1=1$ and $\alpha_2=2$}
\label{fig:c2different}
\end{minipage} \hfill
\begin{minipage}[b]{.5\linewidth}
\centering    \includegraphics[height=0.25\textheight]{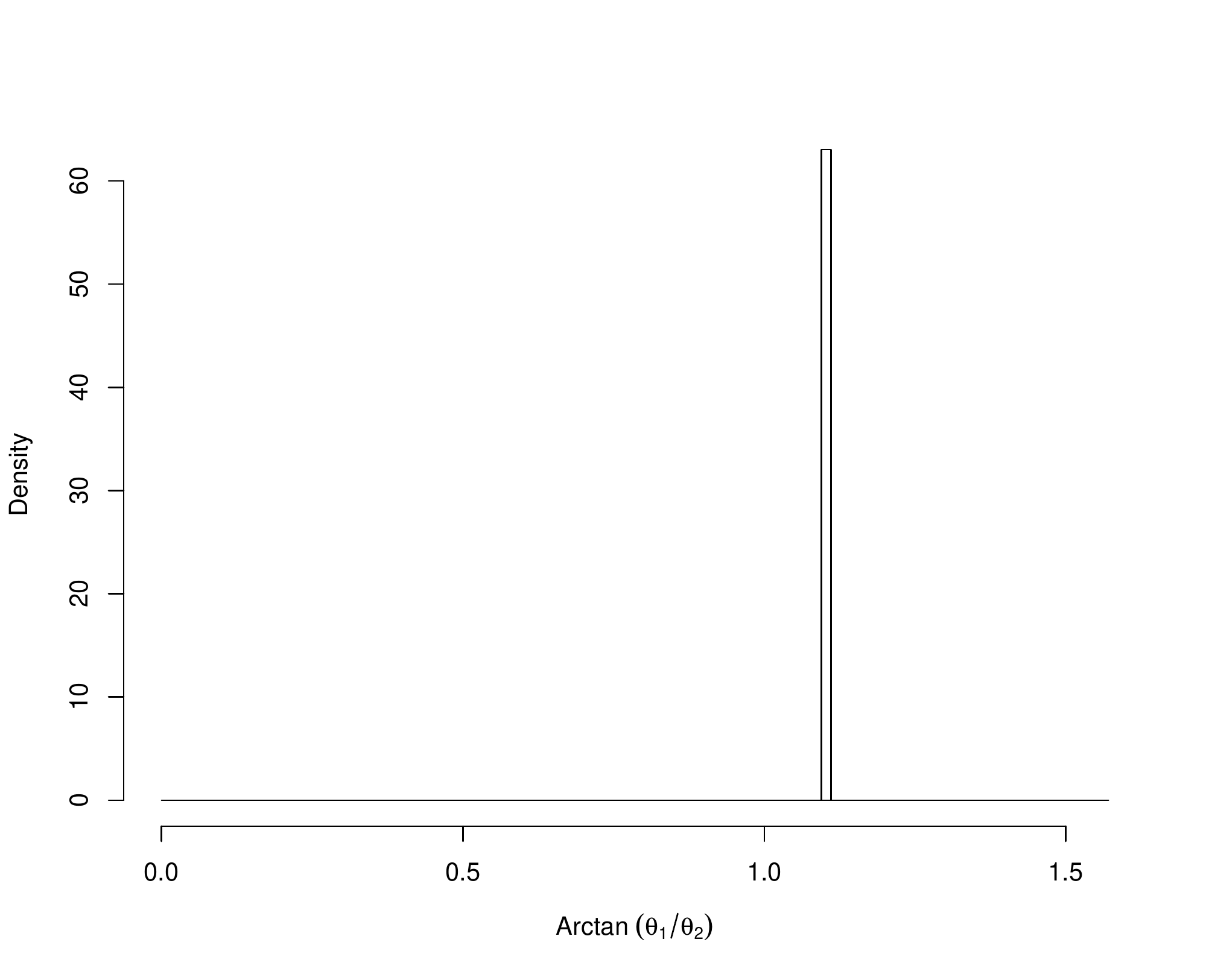}
\caption{\footnotesize  $\alpha_1=\alpha_2=1$}
\label{fig:c2same}
\end{minipage}
\end{figure}

Next we consider \ref{case2'} with $M_t=Z_t^2$, $Z_t\sim \mathcal N(0,1)$, $b_1=b_2=0.1$ and constant $\bb Q_t = (0.2,0.1)^\top$.
This corresponds to the totally correlated CCC-GARCH model. Again the angular measure is concentrated around $0$ and $\pi/2$ in Figure \ref{fig:c2different} for different  coefficients $c_1=0.9$ and $c_2$ the positive root of $3c_2^2+0.2c_2+0.01=1$. For equal coefficients $c_1=c_2=0.9$ in Figure  \eqref{fig:c2same}, the support of the angular measure of $\widetilde{ \bb \Theta}_0$ is concentrated around $\arctan(Q_{t,1}/Q_{t,2})=\arctan(2)$ as expected from Lemma \ref{lem:support spectral measure} (b).

\begin{figure}[h!]
\begin{minipage}[b]{.5\linewidth}
\centering \includegraphics[height=0.25\textheight]{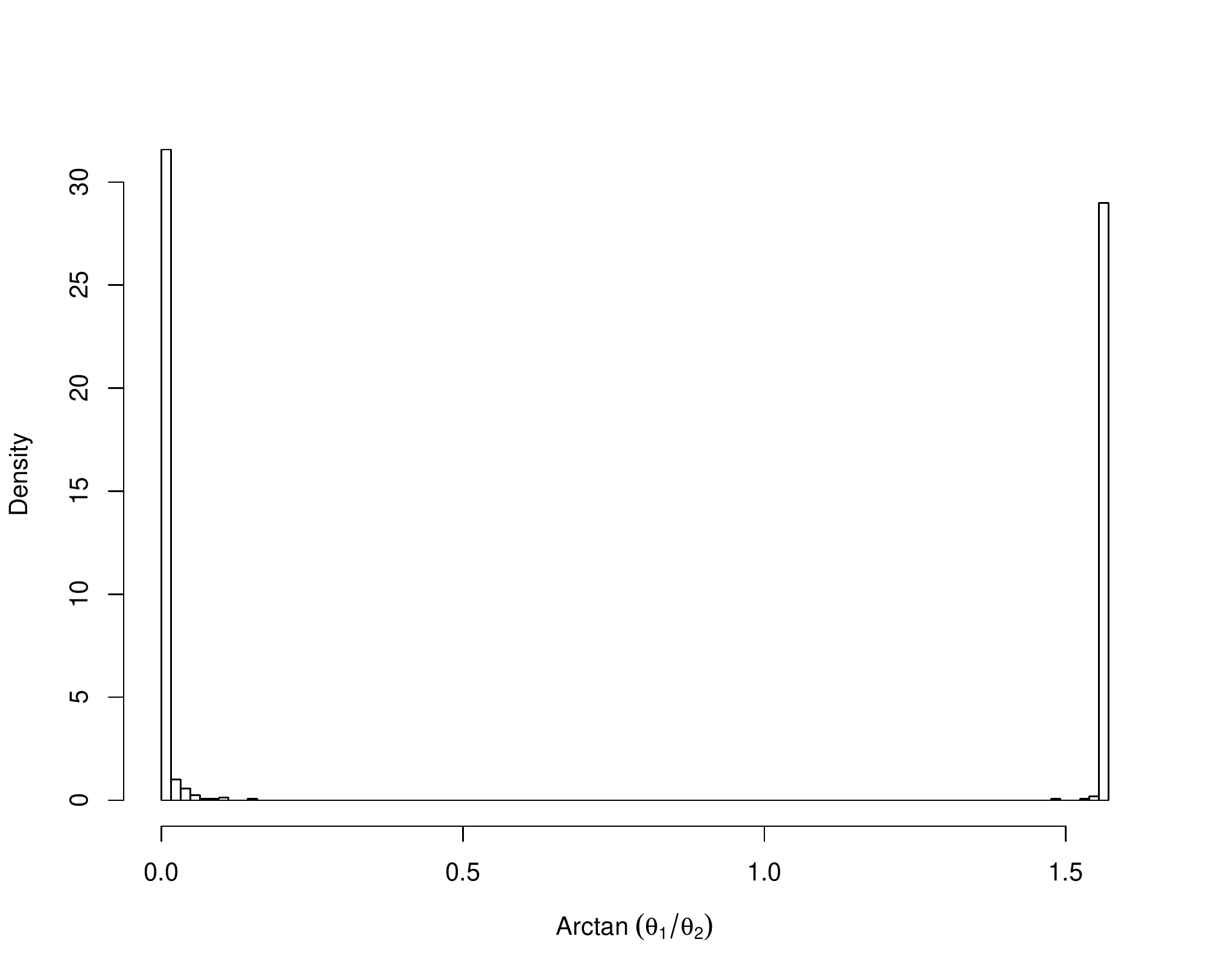}
\caption{\footnotesize  $\alpha_1=1$ and $\alpha_2=2$}
\label{fig:c3different}
\end{minipage} \hfill
\begin{minipage}[b]{.5\linewidth}
\centering    \includegraphics[height=0.25\textheight]{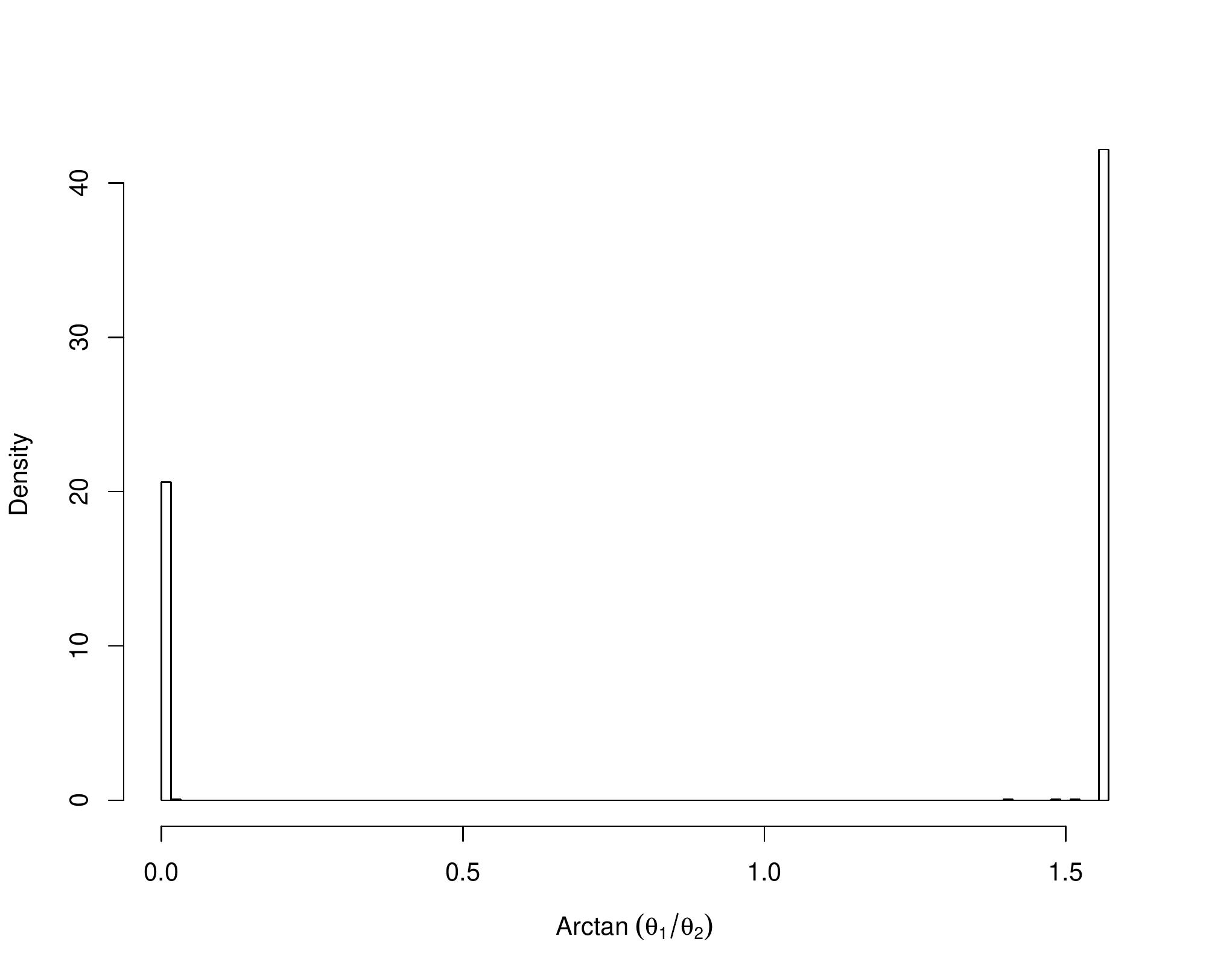}
\caption{\footnotesize  $\alpha_1=\alpha_2=1$}
\label{fig:c3same}
\end{minipage}
\end{figure} 

Figures \ref{fig:c2different}  and \ref{fig:c2same} correspond to the totally correlated CCC-GARCH. In order to investigate our conjecture we consider the case where the covariance matrix $\bb \Sigma$ is standardized with correlation $0.5$.  In Figures \ref{fig:c3different}  and \ref{fig:c3same} the parameters are the same as in Figures \ref{fig:c2different}  and \ref{fig:c2same}. Both angular measures are concentrated around $0$ and $\pi/2$, supporting our conjecture and encouraging us to extend it: The asymptotic independence might be the rule for any CCC-GARCH model except when the CCC-GARCH model is totally correlated with equal marginal tail index.

\appendix

\section*{Appendix}
\section{Proof of Lemma \ref{lem:eqcond}}
\setcounter{stepnumber}{0}

The fundamental ingredient in the proof is a large deviation result for $T_u$ by \cite{Buraczewski.etal:2016a} (see also \cite{buraczewski2018pointwise}). It gives a very precise bound on the typical range of $T_u$, which allows us to deduce properties of the relevant random variables at time $T_u$, by replacing the random time by a deterministic bound.

\medskip

\step Fix $\epsilon >0$ and write $C_u=C_u(\epsilon)$. It is enough to show
 that $\lim_{u\to \infty}\P(C_u^c \, | \, \Tu<\infty)=0$. Indeed, we can  sandwich the conditional probabilities as follows
\begin{align*}  
\lefteqn{\P \big(  X_2>u^{1/\alpha_2}  \, \big| \, \Tu < \infty \big)}\\
&\ge~ \P \big( \big\{X_2>u^{1/\alpha_2} \big\} \cap  C_u  \, \big| \, \Tu < \infty \big)\\
&=~ \P\big( X_2 > u^{1/\alpha_2} \, \big| \, \Tu < \infty\big)-\P\big(\big\{ X_2 > u^{1/\alpha_2} \big\} \cap  C_u^c\, \big| \, \Tu < \infty\big) \\
&\ge~  \P\big( X_2 > u^{1/\alpha_2} \, \big| \, \Tu < \infty\big) -\P\big(  C_u^c\, \big| \, \Tu < \infty\big)\,.
\end{align*}
Then the desired result follows by letting $u\to \infty$.
We will consider each of the contributions to $C_u^c$ separately:
\begin{align*} C_u^c&~=~ \Big\{ {\Tu} > L_u\Big\}\cup \Big\{ {X_{\Tu,1}} > u^{(1 + \epsilon)/\alpha_1} \Big\}  \cup\Big\{\max_{1\le k\le L_u}\dfrac{Q_{k,2}}{Q_{k,1}} >  u^{\varepsilon/\alpha_1}\Big\}\\&   \cup \Big\{ \sum_{l=1}^{\Tu}(W_{2,l}-W_{1,l})-T_u(\mu_{2|1}-\mu_{1|1})> \epsilon\Tu \Big\}   \cup \Big\{ \sum_{l=1}^{L_u} W_{l,1} > \frac{1+\epsilon}{\alpha_1} \log u\Big\}
\\
&~=~ A \cup B \cup D \cup E \cup F.\end{align*}
By \eqref{eq:kg}, the required assertion $\lim_{u \to \infty} P(B|\Tu<\infty)=0$ will as well follow from 
$$ \lim_{u \to \infty} u \cdot P(B \cap \{\Tu < \infty\}) \le \lim_{u \to \infty} u \cdot P(B)=0.$$

\step The negligibility of $A$ is a direct consequence of  \cite[Lemma 4.3]{Buraczewski.etal:2016a}) which provides that for a sufficiently large constant $C$, 
\begin{equation*}
 \lim_{u \to \infty} \P \bigg( \Big| {\Tu} - \frac{\log u}{\mu_1 \alpha_1} \Big| \ge C f(u)\\, \bigg| \, {\Tu} < \infty \bigg)=0\,,
\end{equation*}
where $f(u)= \sqrt{ \log(u) \cdot \log( \log(u)) }.$

\medskip

\step  Negligibility of $B$ and $F$: Considering $B$, we have by \eqref{eq:tailsXstar}  that
 $\lim_{u \to \infty} u \P(X_1 > u^{(1 + \epsilon)/\alpha_1})=0$ implying that
$$\lim_{u \to \infty} u \P(X_{\Tu,1} > u^{(1 + \epsilon)/\alpha_1})=0,$$ since $X_1 = \sup_{n} X_{n,1}$.

By the classical Cram\'er estimate for the random walk $W_{n,1}$ (see \cite[XII.(5.13)]{Feller:1971} it holds
$$ \lim_{u \to \infty} u^{1/\alpha_1} \P\bigg( \sup_{n \in \N} \Big(\sum_{l=1}^n W_{l,1}\Big) > \log u\bigg) = c^* \in (0,\infty)$$
and hence in particular
$$\lim_{u \to \infty} u  \P\bigg( \sum_{l=1}^{L_u} W_{l,1} > \frac{1+\epsilon}{\alpha_1}\log u\bigg)  =0.$$
%

\step Now we turn to $D$. 
A union bound yields
\begin{align*}
\P\Big(\max_{1\le k\le L_u}&\dfrac{Q_{k,2}}{Q_{k,1}} > u^{\varepsilon/\alpha_1}, \, \Tu < \infty \Big)
\le\sum_{k=1}^{L_u}\P\Big( u^{\varepsilon/\alpha_1} Q_{k,1} <  Q_{k,2}, \, \Tu < \infty\Big) \,.
\end{align*}
We decompose for any $k\ge 0$
\begin{align*}
\lefteqn{\P(u^{\varepsilon/\alpha_1} Q_{k,1} <  Q_{k,2},\, {\Tu}<\infty)}\\&=~ \P(u^{\varepsilon/\alpha_1} Q_{k,1} <  Q_{k,2},X_1>u^{1/\alpha_1})\\
&\le~ \P\Big(u^{\varepsilon/\alpha_1} Q_{k,1} <  Q_{k,2}, \\
&~~~ \sum_{j\neq k} \prod_{\ell=1}^{j-1} (b_1+c_1M_\ell) Q_{j,1} +\prod_{\ell=1}^{k-1} (b_1+c_1M_\ell)Q_{k,1}>u^{1/\alpha_1}\Big)\,.
\end{align*}
We bound this probability by the sum of two terms
\begin{align}
 \P\Big(u^{\varepsilon/\alpha_1} Q_{k,1} <  Q_{k,2},\ \sum_{j\neq k}\prod_{\ell=1}^{j-1} (b_1+c_1M_\ell) Q_{j,1}> \frac12{u^{1/\alpha_1}}\Big)  \nonumber \\
 + \P\Big(u^{\varepsilon/\alpha_1} Q_{k,1} <  Q_{k,2},\, \prod_{\ell=1}^{k-1} (b_1+c_1M_\ell)Q_{k,1}>\frac12 u^{1/\alpha_1}\Big) \label{eq:twotermbound}
 \end{align}
 and have to show that both contributions, when summed over $k=0, \dots, L_u$, are of order $o(u^{-1})$.

We estimate the second term in \eqref{eq:twotermbound} thanks to Markov's inequality of order $\alpha_1/(1+\varepsilon)<\kappa<\alpha_1$:
\begin{align*}
\P\Big(u^{\varepsilon/\alpha_1} Q_{k,1} <  Q_{k,2},\ &\prod_{\ell=1}^{k-1} (b_1+c_1M_\ell)Q_{k,1}> \frac12 u^{1/\alpha_1}\Big)\\
&\le \P\Big( \prod_{\ell=1}^{k-1} (b_1+c_1M_\ell) Q_{k,2}> \frac12 u^{(1+\varepsilon)/\alpha_1}\Big)\\
&\le \frac{2^\kappa \big(\E[(b_1+c_1M)^{\kappa}]\big)^k \, \E [\|\bb Q\|^\kappa]}{ u^{\kappa((1+\varepsilon)/\alpha_1)}}\,.
\end{align*}
As $\alpha_1/(1+\varepsilon\alpha_1)<\kappa<\alpha_1$ we have that $q:=\E[(b_1+c_1M)^{\kappa}]<1$ and conclude
\begin{align*}
&~\sum_{k=0}^\infty\P\Big(u^{\varepsilon/\alpha_1} Q_{k,1}< Q_{k,2},\prod_{\ell=1}^{k-1} (b_1+c_1M_\ell)Q_{k,1}>\frac12 u^{1/\alpha_1}\Big)
\\ \le&~\frac{2^\kappa \E [\| \bb Q \|^\kappa]}{1-q} \frac{1}{u^{\kappa((1+\varepsilon)/\alpha_1)}}~=~o(u^{-1})\,.
\end{align*}

\step Finally we turn to $E$. Note that $E = \emptyset$ in \ref{case1}, since then 
$$ W_{l,2}-W_{l,1} ~=~\log \bigg( \frac{c_2M_l}{c_1M_l} \bigg) ~=~ \log c_2 - \log c_1 ~=~ \mu_{2|1} - \mu_{1|1} \quad \text{ a.s.}$$ 
Hence, we consider only \ref{case2}. We have to prove the conditional large deviation result
$$
\lim_{u\to \infty}\P\Big(\underbrace{\sum_{k=1}^{\Tu}(W_{k,2}-W_{k,1}) -\Tu (\mu_{2|1}-\mu_{1|1})>\epsilon \Tu}_{=:D_u}\,\mid\, \Tu<\infty\Big)=0
$$

Under $\P^{\alpha_1}$ we have that $S_n=\sum_{k=1}^n W_{k,1}$ constitutes a random walk with positive drift $\mu_{1|1}=\E^{\alpha_1}[\log(b_1+c_1 M)]>0$. Thus under the change of measure $S_n\to \infty$, $\Tu<\infty$  a.s. Note here that $\Tu$ is not a stopping time for the random walk, but for the sequence $X_{1,n}$. However, divergence of $S_n$ implies divergence of $X_{1,n}$, see e.g. \cite[Theorem 2.1]{Goldie:Maller:2000}. Hence we have the identity
$$
\P\big(D_u\,\big|\, \Tu<\infty\big)=\dfrac{\E^{\alpha_1}\big[e^{-\alpha_1 S_{\Tu}}\, 1_{D_u}\big]}{\P(\Tu<\infty)}\,.
$$
Since $u\P(\Tu<\infty)\to c>0$ as $u\to \infty$, it is enough to show that 
$$
\lim_{u\to\infty}u\E^{\alpha_1}\big[e^{-\alpha_1 S_{\Tu}}\, 1_{D_u}\big]=0\,.
$$
We have 
\begin{align*}
u\E^{\alpha_1}\big[e^{-\alpha_1 S_{\Tu}}\,1_{D_u}\big]&=\E^{\alpha_1}\Big[\Big(\dfrac{u^{1/\alpha_1}}{X_{\Tu,1}}\Big)^{\alpha_1}\Big(\dfrac{X_{\Tu,1}}{e^{S_{\Tu}}}\Big)^{\alpha_1}\,1_{D_u}\Big]\\
&\le\E^{\alpha_1}\Big[\Big(\dfrac{X_{\Tu,1}}{e^{S_{\Tu}}}\Big)^{\alpha_1}\,1_{D_u}\Big]
\end{align*}
by definition of $\Tu$ implying that $X_{\Tu,1}>u^{1/\alpha_1}$. Then, using Chernoff's device, we achieve for any $\lambda>0$ the following upper-bound
\begin{align*}
\E^{\alpha_1}\Big[\Big(\dfrac{X_{\Tu,1}}{e^{S_{\Tu}}}\Big)^{\alpha_1}\,1_{D_u}\Big]\le \E^{\alpha_1}\Big[\Big(\dfrac{X_{\Tu,1}}{e^{S_{\Tu}}}\Big)^{\alpha_1}e^{\lambda (\sum_{k=1}^{\Tu}(W_{k,2}-W_{k,1}) -\Tu (\mu_{2|1}-\mu_{1|1})-\epsilon \Tu)}\Big]\,.
\end{align*}
We will show that 
$$
Z_n(\epsilon,\lambda):= \Big(\dfrac{X_{n,1}}{e^{S_n}}\Big)^{\alpha_1}\dfrac{e^{\lambda \sum_{k=1}^{n}(W_{k,2}-W_{k,1}) }}{\big(e^{\lambda((\mu_2-\mu_1)+\epsilon)}\big)^n}
$$
is an integrable stochastic process with $\sup_{n\ge 1}\E^{\alpha_1}[Z_n(\epsilon,\lambda)]<\infty$ a.s. for any $\epsilon>0$ small by choosing $\lambda>0$ accordingly. Then  $\E^{\alpha_1}[Z_n(\epsilon',\lambda)]\to 0$ for any $\epsilon'>\epsilon$ which is the desired result. \\

In order to show the uniform bound for $\E^{\alpha_1} [Z_n]$ one has to introduce the sequence
$$
V_n=\dfrac{X_{n,1}}{e^{S_n}}=\dfrac{\sum_{k=1}^{n}\prod_{l=1}^{k-1}A_lQ_{k,1}}{\prod_{l=1}^{n}A_l}=\dfrac1{A_n}V_{n-1}+\dfrac{Q_{n,1}}{A_n}
$$
with $A_n=b_1+c_1M_n$. We abbreviate $\Delta W_k:=W_{k,2}-W_{k,1}$ and $\Delta \mu := \mu_{2|1}-\mu_{1|1}$ and note that
$\E^{\alpha_1} (\Delta W_1) ~=~ \Delta \mu$ and that, due to the assumptions of \eqref{case2}
\begin{align*}
d_*~:=~\log(b_2)- \log(b_1) ~\le~ \Delta W_1 ~&=~ \log \Big( \frac{b_2+c_2M_1}{b_1+c_1M_1} \Big) \\
&\le~ \log(c_2)- \log(c_1) ~=:~ d^*.
\end{align*}
We will use the recursive formula
\begin{align*}
Z_n(\epsilon,\lambda)&=V_n^{\alpha_1}\dfrac{e^{\lambda \sum_{k=1}^{n} \Delta W_k }}{\big(e^{(\lambda \Delta \mu+\epsilon)}\big)^n}\\
&=\Big(\dfrac1{A_n}V_{n-1}+\dfrac{Q_{n,1}}{A_n}\Big)^{\alpha_1}\dfrac{e^{\lambda \sum_{k=1}^{n-1}\Delta W_k }}{\big(e^{\lambda \Delta \mu+\epsilon}\big)^{n-1}} \cdot \dfrac{e^{\lambda \Delta W_n}}{e^{\lambda(\Delta \mu+\epsilon)}}\,.
\end{align*}
Then 
\begin{align}
& \E^{\alpha_1}\big[Z_n(\epsilon,\lambda)\mid \mathcal F_{n-1}\big] \label{eq:boundZn}\\
\le~&\E^{\alpha_1}\Big[ \dfrac1{A_n^{\alpha_1}}\dfrac{e^{\lambda(\Delta W_n)}}{e^{\lambda(\Delta \mu+\epsilon)}}\Big]Z_{n-1}+\E^{\alpha_1}\Big[\Big(\dfrac{Q_{n,1}}{A_n}\Big)^{\alpha_1}\dfrac{e^{\lambda\Delta W_n}}{e^{\lambda(\Delta \mu+\epsilon)}}\Big]\dfrac{e^{\lambda \sum_{k=1}^{n-1}\Delta W_k }}{(e^{\lambda \Delta \mu+\epsilon})^{n-1}} \nonumber
\end{align}
and we are going to prove that both the factors 
$$ \mathfrak{m}(\lambda) ~:=~ \E^{\alpha_1}\Big[ \dfrac1{A_n^{\alpha_1}}\dfrac{e^{\lambda \Delta W_n}}{e^{\lambda(\Delta \mu+\epsilon)}}\Big] ~=~ \E \Big[ \dfrac{e^{\lambda\Delta W_1}}{e^{\lambda(\Delta \mu+\epsilon)}}\Big] $$
$$ \mathfrak{c}(\lambda) ~:=~ \E^{\alpha_1}\Big[ \dfrac{e^{\lambda\Delta W_1}}{e^{\lambda( \Delta \mu+\epsilon)}}\Big]$$
are less than one for suitably small $\lambda$. Upon taking $\P^{\alpha_1}$-expectations in \eqref{eq:boundZn}, we infer
$$ \E^{\alpha_1}\big[Z_n(\epsilon,\lambda)\big] ~\le~ \mathfrak{m}(\lambda) \E^{\alpha_1}\big[Z_{n-1}(\epsilon,\lambda)\big] + \mathfrak{c}(\lambda)^{n-1} C_Q$$
with
$$ C_Q = \E^{\alpha_1}\Big[\Big(\dfrac{Q_{n,1}}{A_n}\Big)^{\alpha_1}\dfrac{e^{\lambda\Delta W_n}}{e^{\lambda( \Delta \mu+\epsilon)}}\Big] = \E\Big[Q_{n,1}^{\alpha_1} \dfrac{e^{\lambda\Delta W_n}}{e^{\lambda( \Delta \mu+\epsilon)}}\Big] \le \E\big[Q_{n,1}^{\alpha_1} \big]\dfrac{e^{\lambda d^*}}{e^{\lambda( \Delta \mu+\epsilon)}}<\infty.$$

\medskip

Using again the boundedness of $\Delta W_1$, an application of Hoeffding's lemma yields that
$$\E^{\alpha_1}\big[ {e^{\lambda \Delta W_1}}\big] ~\le~ \exp \Big(\lambda (\Delta \mu) + \frac{\lambda^2}{8}(d^*-d_*) \Big)$$
and hence
\begin{align*}
\mathfrak{c}(\lambda) ~\le~ \exp \Big( \frac{\lambda^2}{8}(d^*-d_*) - \lambda \epsilon \Big).
\end{align*}
Thus $\mathfrak{c}(\lambda)<1$ for suitably small $\lambda>0$.

\medskip

Turning to $\mathfrak{m}(\lambda)$, we again use Hoeffding's lemma to get that
$$ \mathfrak{c}(\lambda) ~\le~ \exp \Big( \lambda \E[\Delta W_1] + \frac{\lambda^2}{8} (d^*-d_*) - \lambda \E^{\alpha_1} [\Delta W_1] - \lambda \epsilon \Big)$$
Hence, it suffices to show that difference of the expectations is nonpositive.
We write
$$ \Delta W_1 ~=~ \log\Big( \frac{b_2 +c_2 M_1}{b_1 + c_1 M_1} \Big) ~=~ \log f(A_1)$$
for $A_1=b_1 +c_1 M_1$ and
$$ f(a) ~=~ \frac{b_2 + c_2 \big( \tfrac{a-b_1}{c_1}  \big)}{a} ~=~- \frac{b_1}{a} \Big( \frac{c_2}{c_1}-\frac{b_2}{b_1} \Big) + \frac{c_2}{c_1}, \qquad a \ge b_1.$$
The function $f$ is increasing and $\log f(a) \ge 0$ for $a \ge b_1$.
Then
\begin{align*}
 &~\E[\Delta W_1] -\E^{\alpha_1} [\Delta W_1] ~=~ \E \Big[ \log f(A_1) \big(1-A_1^{\alpha_1}\big) \Big]  \\
 ~=&~ \E \Big[ \log f(A_1) \big(1-A_1^{\alpha_1}\big)\mathbf{1}_{\{A \le 1\}}  \Big] + \E \Big[ \log f(A_1) \big(1-A_1^{\alpha_1}\big)\mathbf{1}_{\{A >1\}}  \Big] 
\end{align*}
Recall that $\E[A_1^{\alpha_1}]=1$, thus
\begin{align*} -\E \big[(1-A_1^{\alpha_1}) \mathbf{1}_{\{A >1\}}\big] ~&=~ \E \big[(1-A_1^{\alpha_1})\mathbf{1}_{\{A \le1\}}\big] \\
 -\E \big[\log f(A_1)(1-A_1^{\alpha_1})\mathbf{1}_{\{A >1\}}  \big] ~&\ge ~ \E \big[\log f(A_1)(1-A_1^{\alpha_1})\mathbf{1}_{\{A \le 1\}} \big] \end{align*}
where we have used that the function $f$ is increasing and nonnegative.

\end{document}